\documentclass[a4paper,12pt]{amsart} 
\usepackage{tikz}
\usetikzlibrary{arrows.meta}
\usetikzlibrary{positioning}
\usetikzlibrary{calc}
\usetikzlibrary{decorations.pathreplacing}

\usepackage{graphicx}
\usepackage{caption}
\usepackage{subcaption}

\usepackage{bm}

\usepackage{soul}
\usepackage{amsmath, amssymb, amsfonts, amsthm, amssymb}
\usepackage{euscript, mathrsfs, stackrel}
\usepackage{enumerate}
\usepackage{stmaryrd}
\usepackage[utf8]{inputenc}   % LaTeX, comprends les accents !
\usepackage[T1]{fontenc}      % Police contenant les caractère français
\usepackage{ulem} %barrer

\usepackage{color}
\usepackage{graphicx}
\usepackage{hyperref}
\usepackage{wrapfig}
\usepackage{cleveref}

\theoremstyle{plain}

\numberwithin{equation}{section}

\newtheorem{theorem}{Theorem}[section]

\newtheorem{proposition}[theorem]{Proposition}
\newtheorem{lemma}[theorem]{Lemma}

\newtheorem{corollary}[theorem]{Corollary}

\usepackage[left=2.5cm,right=2.5cm,vmargin=2.5cm]{geometry}

\usepackage{mathrsfs} %%% Nice caligraphic fonts

\def\r{\mathbb R}
\def\e{\mathbb E} %%expectation wrt everything%%
\def\p{\mathbb P} 
\def\lb{\llbracket}
\def\rb{\rrbracket}
\def\V{{\mathcal V}}
\def\H{{\mathcal H}}

\def\L{{\mathscr L}}

\def\cvp{{\buildrel \mbox{\tiny\rm (p)} \over \longrightarrow}}
\def\cvL2{{\buildrel \mbox{\tiny\rm ($L^2$)} \over \longrightarrow}}

\title{Average-weight percolation on the complete graph}

\author{Elie A\"id\'ekon}
\address{Elie A\"id\'ekon, SMS, Fudan University, China}
\email{aidekon@fudan.edu.cn}

\author{Yueyun Hu}
\address{Yueyun Hu, 
LAGA, Universit\'e Paris XIII,  93430 Villetaneuse,
France}
\email{yueyun@math.univ-paris13.fr}

\date{\today}

\begin{document}

\begin{abstract} 
Attach to each edge of the complete graph on $n$ vertices, i.i.d. exponential random variables with mean $n$. Aldous \cite{aldous} proved that the longest path with average weight below $p$ undergoes a phase transition at $p=\frac{1}{e}$: it is  $o(n)$ when $p<\frac{1}{e}$ and of order $n$ if $p>\frac1e$. Later, Ding \cite{ding} revealed a finer phase transition around $\frac{1}{e}$:   there exist $c'>c>0$ such that the length of the longest path is of order $\ln^3 n$ if $ p \le \frac{1}{e}+\frac{c}{\ln^2 n}$ and is  polynomial  if $p\ge  \frac{1}{e}+\frac{c'}{\ln^2 n}$.  We identify the location of this phase transition and obtain sharp asymptotics of the length near criticality.  The proof uses an exploration mechanism mimicking a branching random walk with selection introduced by Brunet and Derrida \cite{brunet_derrida2}.
\end{abstract}

\subjclass[2010]{60C05,60F05,  60G70,  60J80}

\keywords{Average-weight percolation, branching random walk with selection.}

\maketitle

\section{Introduction}

We consider the complete graph $K_n$ of the set of vertices $\lb 1, n\rb=\{1,2,\ldots,n\}$. Following Aldous \cite{aldous}, we  assign each edge  an independent exponential random variable with mean $n$. In this setting, Aldous proved the existence of a phase transition for the size of  the largest tree with average weight below a given threshold \cite[Theorem 1]{aldous}. We are interested in the analogous result for paths, rather than trees, established in  \cite[Proposition 3.7]{aldous}. In order to state this result, we introduce the following notation. 
For ease of notation, in this paper, we will suppose that the labels $X(u,v)_{u\neq v \in \lb 1, n\rb}$  on the (non-oriented) edges of $K_n$ are i.i.d. copies of 
\begin{equation}\label{def:X}
X=ne \mathrm{Exp}(1)-1
\end{equation} 
where $\mathrm{Exp(1)}$ denotes an exponential random variable with mean $1$. 
A path $\gamma$ is a sequence of vertices $(\gamma_0,\gamma_1,\ldots,\gamma_\ell) \in  \lb 1, n\rb^{\ell+1}$ which is made of distinct adjacent edges $(\gamma_i,\gamma_{i+1})$. We call $|\gamma|=\ell$ the length of the path, and let ${\mathcal P}_\ell$ be the collection of paths of length $\ell$. A path of length $0$ is reduced to a single vertex.

For any $\gamma \in {\mathcal P}_\ell$, we let $S_0(\gamma)=0$ and 
\begin{equation}    
S_i(\gamma) := X(\gamma_0,\gamma_1)+\ldots+X(\gamma_{i-1},\gamma_i),\qquad  i\in \lb 1,  \ell\rb. \label{Sigamma}
\end{equation}
We write $S(\gamma):=(S_1(\gamma),\ldots,S_\ell(\gamma))$ and call $\frac{S_{\ell}(\gamma)}{\ell}$ the average weight of $\gamma$. For $\lambda\in (-1, \infty)$, let $\L(n,\lambda)$ be the length of the longest path with average weight below $\lambda$:
\begin{equation}    
\L(n,\lambda):=\sup\{ \ell \ge 1\colon \exists \, \gamma \in  {\mathcal P}_\ell \textrm{ such that } S_{\ell}(\gamma) \le \lambda \ell \}.
\label{def-Lnlambda} \end{equation}

  Translated in our setting, \cite[Proposition 3.7]{aldous} states that $\L(n, \lambda)=o(n)$ if $\lambda<0$ and  $\L(n,\lambda)=\Omega(n)$ if $\lambda>0$. 
 Later, Ding \cite{ding} proved that there is a critical window of size $(\ln n)^{-2}$ around $\lambda=0$:  there exist some constants $c'> c >0$ and $C'>C$ such that 
 
 (i) if $ -1\le \alpha \le c $, then  \begin{equation} \p\Big( C \ln^3 n \le \L\big(n,\frac{\alpha}{\ln^2 n}\big)  \le C' \ln^3 n\Big) \to 1, \qquad n \to \infty;  \label{Ding-1}\end{equation}

(ii) if $\alpha \ge c'$, then \begin{equation} \p\Big( n^{1/4} \le \L\big(n,\frac{\alpha}{\ln^2 n}\big)  \le C' \frac{n}{\ln^2 n}\Big) \to 1, \qquad n \to \infty.  \label{Ding-2}\end{equation}

The constant $-1$ in (i) is arbitrary, see \cite[Theorem 1.3]{ding}. To the best of our knowledge, it remains open to identify the critical value at which  the longest path undergoes a transition from  order $\ln^3 n$ to   polynomial order, as well as the precise asymptotic behavior of $\L\big(n,\frac{\alpha}{\ln^2 n}\big)$. Another natural problem concerns the near-critical case $\lambda = \varepsilon$ for small $\varepsilon > 0$, which was further investigated by Ding and Goswami \cite{ding_goswami}: there exists some $\varepsilon_0>0$ such that for all $\varepsilon < \varepsilon_0$, \begin{equation} \p\Big( n e^{- C' \varepsilon^{-1/2}} \le \L(n, \varepsilon) \le   n e^{- C \varepsilon^{-1/2}}\Big) \to 1.  \label{Ding-Goswami}\end{equation}

  The following theorems give sharp asymptotics of $\L(n,\lambda)$ in these regimes.  
\begin{theorem}\label{t:main}
If $\alpha<\frac{\pi^2}{2}$, then as $n\to\infty$, $$\frac{1}{\ln^3 n} \L\big(n,\frac{\alpha}{\ln^2 n}\big) \, \cvp\, \frac{1}{\frac{\pi^2}{2}-\alpha}.$$ 
\end{theorem}

Note that taking $\alpha=0$ in Theorem \ref{t:main} answers \cite[{Equation (2)}]{ding}.

\begin{theorem}\label{t:main2}
 For each $\varepsilon>0$, let $\beta_n$ be a sequence with $\beta_n \in (0,\varepsilon)$ for all $n$ and such that $\liminf_{n\to\infty}  \frac{\sqrt{2\beta_n}}{\pi} \ln n > 1$.  As $n\to \infty$ then $\varepsilon\to 0$, $$\sqrt{\beta_n}\Big(\ln \L(n,\beta_n) -\ln n\Big) \, \cvp\,  -\frac{\pi}{\sqrt{2}}.$$
%$\sqrt{\beta_n}\Big(\ln \L(n,\beta_n) -\ln n\Big) \to -\frac{\pi}{\sqrt{2}}$ in probability. 
 In other words, for any $a>0$, 
\[
\lim_{\varepsilon\downarrow0}\liminf_{n\to\infty}\p\Big(\frac{1}{n}\L(n,\beta_n) \in [e^{-\frac{\pi}{\sqrt{2\beta_n}}(1+a)}, e^{-\frac{\pi}{\sqrt{2\beta_n}}(1-a)}]\Big)=1.
\]
\end{theorem}

Taking $\beta_n= \frac{\alpha}{\ln^2 n}$ with $\alpha > \frac{\pi^2}{2}$, we see that \begin{equation}    \L\big(n,\frac{\alpha}{\ln^2 n}\big)= n^{1- \frac{\pi}{\sqrt{2 \alpha}} +o_p(1)},\label{alpha:polynomial}\end{equation}

\noindent where $o_p(1)$ denotes some quantity which converges to $0$ in probability as $n\to\infty$. This, together with Theorem \ref{t:main}, shows that there is a phase transition at $\lambda_c=\frac{\pi^2}{2}\frac1{\ln^2 n}$. We also note that \eqref{alpha:polynomial} is consistent with \eqref{Ding-2} when $\alpha$ is taken sufficiently large. 

This critical value $\lambda_c$ also appears in a related problem concerning the minimum average weight of a cycle in $K_n$. 
Ding, Sun and Wilson \cite{ding_sun_wilson} proved that conditionally on the minimum being nonnegative, its value is $\frac{\pi^2}{2 \ln^2 n}  + O(\frac{1}{\ln^3 n})$ and that the length of the corresponding cycle is of order $\ln^3 n$. It would be interesting to investigate further the connection between these two models.

\bigskip
Now, we outline the main ideas of the proofs. 
 Heuristically, we  explore the graph in the following way. Let $N=o(n)$ be an integer. We start with an initial set of $N$ vertices, called generation $0$. We examine all adjacent edges of these vertices
  and select the $N$ edges with the  smallest labels. The endpoints of the selected edges form generation $1$, and the associated labels stand for their positions. We then delete the initial set of vertices and all their adjacent edges, and continue the exploration from the vertices of generation $1$. Each incident edge of a vertex at generation $1$ carries a weight. Add this weight to the position of the vertex, 
  then select the $N$ edges associated to the $N$ smallest obtained numbers. The endpoints of the selected edges form the vertices of generation $2$ and the associated numbers stand for their positions.  This process can be repeated for roughly $\frac{n}{N}$ steps before the graph is exhausted. The weights of the paths obtained by this strategy resemble an (inhomogeneous)  branching random walk  in which only the $N$ leftmost particles are retained at each step. This model was introduced by Brunet and Derrida \cite{brunet_derrida2} who conjectured that the selection mechanism was shifting the speed of the population by a correction term of order $\frac{1}{\ln^2 N}$, a result later proved by B\'erard and Gou\'er\'e \cite{berard_gouere}.   We refer to Maillard \cite{maillard}
  for more  results on the model in continuous time. In our setting, the minimum of the exploration process at time $t$ is of order 
\[
-\ln N + \frac{\pi^2}{2 \ln^2 N} t.
\]
Taking  $N=n^{1-o(1)}$  and $N=e^{\frac{\pi}{\sqrt{2\beta_n}}}$ yields Theorem \ref{t:main} and Theorem \ref{t:main2}, respectively. A slight variation of this strategy is used to establish the lower bound for $\L(n,\lambda)$. The upper bound follows a similar approach to that of \cite{ding,ding_goswami,ding_sun_wilson} by computing suitable truncated first moments. 

\bigskip

Finally, we mention the recent preprint by Jorritsma, Maillard and M\"orters \cite{JMM} in which the authors make use of a killed branching random walk to explore growing sparse random graphs. Even if the models are different, our results  show similarity with several features of  their model, see \cite{JMM} and the references therein.

\bigskip

The paper is organized as follows. Sections \ref{s:upper} and \ref{s:lower} give the proofs of the upper bounds and lower bounds for $\L(n,\lambda)$ respectively.  Section \ref{s:rw-brw} collects results on real-valued random walks and  branching random walks that are used in Sections \ref{s:upper} and \ref{s:lower}.

\section{Proofs of the upper bounds in Theorems \ref{t:main} and \ref{t:main2}}
\label{s:upper}

Let $(S_1,S_2,\ldots)$ denote a random walk with step distribution $X$ of \eqref{def:X}, starting at $S_0=0$. For any $\ell\ge 1$ and Borel set $A\subset \r^\ell$, 
\begin{equation}\label{comp:moment1}
\e\Big[\#\{\gamma \in {\mathcal P}_\ell\colon S(\gamma) \in A\}\Big] \le n^{\ell+1}\p\Big((S_1,\ldots,S_\ell) \in A\Big), 
\end{equation}

\noindent where we used the fact that the number of paths of length $\ell$ in $K_n$ is smaller than $n^{\ell+1}$, and we recall that $S(\gamma)$ denotes the vector whose $i$-th coordinate, $S_i(\gamma)$,  was defined in \eqref{Sigamma}. Let $q_n=1-\frac1{en}$ so that $\e[e^{- q_n X}]=\frac{1}{n}\exp(-\frac1{en})$. The r.v. $X$ tilted by $e^{-q_n X}$ is distributed as $\textrm{Exp}(1)-1$.  We let $Y$ be a random walk starting at $0$ associated with this tilted distribution so that $Y$ is a centered random walk. Notice that 
\begin{equation}\label{eq:many-to-one}
\p\Big((S_1,\ldots,S_\ell) \in A\Big) = n^{-\ell}e^{-\frac{\ell}{ne}}\e\Big[e^{q_n Y_\ell} {\bf 1}_{\{(Y_1,\ldots,Y_\ell) \in A\}}\Big].
\end{equation}

 \begin{lemma}\label{l:inf}
For any $x\ge 0$, $$ \e \Big[ \#  \big\{\gamma \in \cup_{\ell\ge 1} {\mathcal P}_\ell : S_j(\gamma) \ge -x, \forall\, j\le |\gamma|-1, S_{|\gamma|}(\gamma) < -x\big\} \Big]
\le n \, e^{- q_n x}.$$ 
\end{lemma}
\begin{proof} By \eqref{comp:moment1} and \eqref{eq:many-to-one}, we deduce from the union bound that the expectation term in Lemma \ref{l:inf} is bounded above by  $$ \sum_{\ell=1}^\infty n \e\Big[ e^{ q_n Y_\ell} {\bf 1}_{\{ Y_j \ge -x, \forall j\le \ell-1, Y_\ell < -x\}}\Big]$$ which is further bounded by $$n \, e^{- q_n x} \, \sum_{\ell=1}^\infty   \p\Big(  Y_j \ge -x, \forall j\le \ell-1, Y_\ell < -x\Big)= n \, e^{- q_n x}, $$  \noindent where the above sum over $\ell$ is equal to $1$, because  $\ell$ corresponds to the first time at which the centered random walk $(Y_j)$ crosses below $-x$. 
\end{proof}

  We deduce the following.

\begin{corollary}\label{cor:jump} 

{\rm (i)}  Let  $b_n$ be a sequence such that $b_n=\ln n +o(\ln n)$ and $b_n-\ln n\to \infty$. With probability tending to $1$ as $n\to \infty$, there is no path $\gamma$ for which $S_j(\gamma)-S_i(\gamma)\le -b_n$ for some $0\le i \le j \le |\gamma|$.

{\rm (ii)}  Let $\Delta_n=o(n)$ be a sequence of numbers in $(0,\infty)$ and $M> 2$. For all sufficiently large $n$, with probability at least $1- \frac{2}{M}$,  we have that for any  path $\gamma$, the number of indices $i \in \{0, 1, ..., |\gamma|\}$ for which there exists some $j\ge i$ such that  $S_j(\gamma)-S_i(\gamma) \le -\Delta_n$  is at most $Mne^{-\Delta_n}$.

    \end{corollary}
    
 \begin{proof} The proofs rely on the following simple argument, already noted  in \cite{ding,ding_goswami,ding_sun_wilson}.  If a path as in (i) existed, we could extract a subpath $\widetilde{\gamma}$ of some length $\widetilde{\ell}$ such that $S_j(\widetilde{\gamma})\ge -b_n$ for all $0\le j<\widetilde{\ell}$ and $S_{\widetilde{\ell}}(\widetilde{\gamma})< -b_n$. The probability to find such a path $\widetilde\gamma$ goes to $0$ by Lemma \ref{l:inf}. This yields (i).  
 
 For (ii), we proceed in a similar way. Let $r= \lceil Mne^{-\Delta_n} \rceil$. If there is a path $\gamma$ such that there are at least $r$ indices $i$ for which $S_j(\gamma)- S_i(\gamma) \le - \Delta_n$ for some $j\ge i$, then we can extract $r$-subpaths $\widetilde\gamma$ such that $S_j(\widetilde{\gamma})\ge -\Delta_n$ for all $0\le j<\widetilde\ell:=|\widetilde\gamma|$ and $S_{\widetilde{\ell}}(\widetilde{\gamma})< -\Delta_n$. Applying Lemma \ref{l:inf} once again, the expected number of such subpaths $\widetilde\gamma$ is less than $n e^{-q_n \Delta_n}$. The conclusion then follows from Markov's inequality.  
 \end{proof}

In view of Corollary \ref{cor:jump} (i) and \eqref{comp:moment1}, the upper bound for $\L(n,\lambda)$ in Theorem \ref{t:main} is a consequence of the following lemma. 

\begin{lemma}\label{main:upper}  Let   $c>0$ 
be an arbitrary constant. For any $d \in (0, \frac{\pi^2}2)$, there exists $\Delta_0=\Delta_0(c, d)\ge 1$ such that for all $\Delta\ge \Delta_0$ and $L\ge \Delta^{2+c}$, we have 
\[
 \p\Big( Y_L \le \frac\Delta{c}     ,\, \min_{0\le i  \le j \le L}(Y_j-Y_i) \ge - \Delta \Big) \le  e^{- \frac{d L}{\Delta^2}} .
\]
\end{lemma}

The proof of Lemma \ref{main:upper} is postponed to Section \ref{sub:rw}.

\begin{proof}[Proof of the upper bound in Theorem \ref{t:main}]  Recall by Claim 2.4 of \cite{ding} that if there is a path of length greater than $\ell$ with average weight below $\lambda$, then there is a path of length in $\lb \ell,2\ell\rb$ with average weight below $\lambda$. Then for $\lambda= \frac{\alpha}{\ln^2 n}$,  $t> \frac{1}{\frac{\pi^2}{2}-\alpha}$,  and  $\ell= \lfloor t \ln^3 n\rfloor$, \begin{align}     & \p\Big(\L(n, \lambda) \ge \ell\Big) 
 \nonumber \\
  \le  & \p\Big(\cup_{L=\ell}^{2 \ell}\{ \exists \gamma \in  {\mathcal P}_L: S_L(\gamma) \le \lambda L \}\Big) \nonumber 
\\
 \le & \p\Big(\cup_{L=\ell}^{2 \ell}\{ \exists \gamma \in  {\mathcal P}_L: S_L(\gamma) \le \lambda L,  \min_{0\le i \le j \le L} (S_j(\gamma)- S_i(\gamma)) > -b_n\}\Big) + o(1), \label{eq:L>ell-bn}
\end{align}

\noindent where $b_n:= \ln n + \ln \ln n$ and we have used Corollary \ref{cor:jump} (i) in the second inequality. By the union bound, it is enough to show that there is some positive constant $\upsilon$ such that for all large $n$, uniformly in $L \in \lb \ell, 2 \ell \rb$, $$ a_n:=\p\Big( \exists \gamma \in  {\mathcal P}_L: S_L(\gamma) \le \lambda L,  \min_{0\le i \le j \le L} (S_j(\gamma)- S_i(\gamma)) > -b_n\Big) \le n^{-\upsilon}.$$

By \eqref{comp:moment1} and \eqref{eq:many-to-one}, $$a_n \le n\, e^{ q_n\lambda L} \, \p\Big(Y_L  \le \lambda L,  \min_{0\le i \le j \le L} (Y_j - Y_i) > -b_n\Big). $$

Recall that $q_n= 1-\frac{1}{en}$.  Let $d \in (0, \frac{\pi^2}{2})$ be such that $d>\alpha$ and $\upsilon:= \frac12((d-\alpha)t-1)>0$. Applying Lemma \ref{main:upper}  with $\Delta= b_n$, we get that  $a_n \le n \exp\big(q_n\lambda L - d \frac{L}{b_n^2}\big)$.  Since $q_n\lambda- d b_n^{-2} \sim - (d-\alpha) \ln^{-2} n$, and $L\ge \ell \sim t \ln^3n$, we get that $a_n \le n^{-\upsilon}$, which completes the proof of the upper bound in Theorem \ref{t:main}.
    \end{proof}

In view of Corollary \ref{cor:jump} (ii) and \eqref{comp:moment1}, the upper bound for $\L(n,\lambda)$ in Theorem \ref{t:main2} is a consequence  the following lemma whose proof is postponed to Section \ref{sub:rw}.

\begin{lemma}\label{main2:upper}  Let $c>0$ be an arbitrary constant. For any $d \in (0, \frac{\pi^2}2)$, there exists $\Delta_0=\Delta_0(c, d)\ge 1$ such that for all $\Delta \ge \Delta_0$ and $L\ge \Delta^{2+c}$, we have 
\[
 \p\Big( \max_{j\in \lb 0, L\rb} Y_j \le \frac{L}{\Delta^{1+c}}     ,\, m_\Delta(L) \le \frac{L}{\Delta^{2+c}} \Big) \le  e^{- \frac{d L}{\Delta^2}} ,
\] where $m_\Delta(L)= \#\{j\in \lb 0,L\rb: \inf_{\ell\in \lb j,L\rb} Y_\ell\le Y_j- \Delta\}$.
\end{lemma}

\begin{proof}[Proof of the upper bound in Theorem \ref{t:main2}] Let $n$ be large and  $a \in (0, 1)$. Consider $\ell= \lceil ne^{- \frac{\pi}{\sqrt{ 2 \beta_n}} (1-a)}\rceil$,  $r=e^{\frac{\pi}{\sqrt{2\beta_n}}(1-\frac{a}2)}$ and  let $m'_\Delta(L)$ be the number of $i\in \lb 0,L\rb$ such that $\inf_{j \in \lb i,L\rb} S_j\le S_i- \Delta$ with $\Delta:=\ln r$.  Let $b_n=\ln n + \ln\ln n$.  By the same argument leading to \eqref{eq:L>ell-bn}, but using Corollary \ref{cor:jump} (ii) in addition, we obtain that  for any $M>2$, $$\p\Big(\L(n, \beta_n) \ge \ell\Big) \le  \sum_{L=\ell}^{2\ell} \widetilde a_L + \frac{2}{M}+ o(1),$$

\noindent where for $L\in \lb \ell, 2\ell \rb$, $$ \widetilde a_L:= \p\Big(\exists \gamma\in {\mathcal P}_L: S_L(\gamma) \le \beta_n L, m'_\Delta(L) \le M \frac{n}{r},\, \min_{0\le i \le j \le L} (S_j(\gamma)- S_i(\gamma)) >-  b_n \Big).$$

By \eqref{comp:moment1} and \eqref{eq:many-to-one}, \begin{align*}    \widetilde a_L & \le n \, e^{\beta_n L}  \p\Big( Y_L\le \beta_n L , \, m_\Delta(L) \le M \frac{n}{r},\,  \min_{0\le i \le j \le L} (Y_j - Y_i ) >- b_n\Big)
\\
& \le n \, e^{\beta_n L}  \p\Big(   \max_{i \in \lb 0, L\rb} Y_i \le 2 \beta_n L, \,  m_\Delta(L) \le M \frac{n}{r}\Big)
 ,
\end{align*} where $m_\Delta(L)$ is as defined in Lemma \ref{main2:upper} and in the last inequality, we have used the fact that $\max_{i \in \lb 0, L\rb} Y_i\le Y_L+b_n \le \beta_n L+ b_n \le 2 \beta_n L$. Choose $d\in (0, \frac{\pi^2}2)$ 
  such that $\eta:= \frac{2d}{\pi^2 (1-\frac{a}2)^2} - 1>0$. Fix an arbitrary $c\in (0, 1)$. If $\varepsilon>0$ is chosen small enough in Theorem \ref{t:main2}, we have $L\ge \Delta^{2+c}$, $M \frac{n}{r} \le L \Delta^{-2-c}$ and $2 \beta_n L \le L \Delta^{-1-c}$.  Note that $\Delta$ is arbitrarily large when $\varepsilon$ goes to $0$.   Lemma \ref{main2:upper} yields that $$ \widetilde a_L \le n e^{\beta_n L} e^{- d \frac{L}{\Delta^2}} = n e^{- \eta \,\beta_n L}. $$
 
 Notice that for all large $n$, $\beta_n L \ge \beta_n \ell \ge n^{a'}$ for some positive constant $a'$. Hence   $\sum_{L=\ell}^{2\ell} \widetilde a_L \le n^2  e^{- \eta\, n^{a'}}\to 0$ and we conclude the proof by letting $n \to\infty$ then $M\to\infty$.  \end{proof}

\section{Proofs of the lower bounds}
\label{s:lower}

We consider a branching random walk $\V$ on $\r$. It starts with a certain number of particles at $0$. At each discrete time $k\ge 0$, each particle of the current population splits into an infinite number of children, whose displacements from their parent  are given by a copy of 
\begin{equation} \label{eq:BRW-V}   
\sum_{i=1}^\infty {\delta}_{\{e T_i -1\}} 
\end{equation}

\noindent where $0<T_1<T_2<\ldots$ are the atoms of a Poisson point process of intensity $1$ on $\r_+$. The position of a particle $u$ is denoted by $\V(u)$.
The offspring distribution corresponds to the limiting law  of $\sum_{u=2}^{n} \delta_{X(1,u)}$, so that, when $n$ is large, exploring the graph around a given vertex locally approximates the branching random walk $\V$. To make it precise, we will  couple labels $X(u,v)$ and copies of the Poisson point process $(T_i)_{i\ge 1}$ as follows. 

\medskip

Consider $(X_i)_{1\le i \le n-1}$ a family of $(n-1)$ i.i.d. copies of $X= n e \mbox{Exp}(1)-1$.  Its increasing order statistics $ X_{(0)}:=-1 < X_{(1)}< ...< X_{(n-1)}$ verify that the variables $X_{(i)}-X_{(i-1)}$ are independent, distributed as $\frac{n e }{n-i}\text{Exp}(1)$.   One can couple it with a Poisson point process $T_1<T_2<\ldots$ of intensity $1$, by 
\[ %\begin{equation}\label{def:coupling}
    X_{(i)}-X_{(i-1)} = \frac{n e}{n-i} (T_{i}-T_{i-1}), \qquad i\in \lb 1, n-1\rb, 
\]%\end{equation}
where we set $T_0=0$.  In particular,  for all $i\in \lb 1, n-1\rb$, \begin{equation}\label{coupling}
   e T_i-1 \le  X_{(i)} \le   \frac{n}{n-i} (e T_{i}-1) + \frac{i}{n-i}.
\end{equation}

For each $u\in K_n$, this gives a coupling between $(X(u,v))_{v\neq u}$ and a copy of $(T_i)_{i\ge 1}$.  We further need the following simple observation.  Let $m \in \lb 1, n-1 \rb$. If we choose   $m$  points uniformly from $(X_{(i)})_{1\le i \le n-1}$ and arrange them in increasing order, then they have the same distribution as  the increasing order statistics of $(X_i)_{1\le i \le m}$. In other words, there exists a coupling in the sense of \eqref{coupling}, between $X(u, v)$ when $v$ varies over a subset of $K_n$ with $m$ vertices, and the family of $m$  points  chosen uniformly from $(T_i)_{1\le i \le n-1}$.

Let $N$ be an integer between $1$ and $n-1$.  We denote by $\V^N$ the branching random walk $\V$  in which, at each step, we only keep the $N$ leftmost particles of the next generation. For a particle $u$ in $\V^N$, let   $|u|$ denote its generation in the genealogical tree of $\V^N$, and let $\V^N(u)$ denote its position in $\r$.

\begin{proposition}\label{p:first_moment} Fix $a\in [0, 1]$ and $t>0$.    If $\V^N$ starts with $\lfloor N^a\rfloor$ particles located at $0$, then for any sequence $\ell_N$ satisfying $\frac{\ell_N}{\ln^3 N} \to t$,  \begin{equation}  \frac{1}{\ln N}  \min_{|u|=\ell_N} \V^N(u)   \,\cvL2\, -a+\frac{\pi^2}{2}t  , \qquad N \to\infty. \label{Nbrw-cvL2}\end{equation}
\end{proposition}

We will apply the above proposition to two cases: $a=1$ and $a=0$.
The proof of Proposition \ref{p:first_moment} is postponed to Section \ref{sub:brw}.

In what follows, we construct a coupling between $\V^N$ and an exploration process on $K_n$. 
We will color particles of $\V^{N}$ in blue, red or purple. Blue particles will correspond to  vertices of $K_{n}$ and will be the particles of interest. 

\subsection{Construction of an exploration process $(Z_j, {\mathcal Q}_j)_{j\ge 0}$ in $K_n$ and its coupling with $\V^N$.}
 \label{s:exploration}
 
For each $j\ge 0$, we define a  set $Z_j$ of so-called {\it active} vertices of $K_n$ and a set of paths ${\mathcal Q}_j\subset {\mathcal P}_j$  such that each vertex  in $Z_j$ is associated with a unique  path in ${\mathcal Q}_j$ ending at this point. The set $Z_j$ will correspond to the blue particles of $\V^{N}$ at time $j$. We will therefore write $u$ either for the blue particle or its corresponding  vertex in $K_n$.  For each active vertex $u$ at step $j$, we will write $V(u)$  for $S_j(\gamma)$ where $\gamma$ is the path in ${\mathcal Q}_j$ ending at $u$.

Let $G_0$ be a complete subgraph of $K_n$. Let $Z_0\subset G_0$ be a set of $N_0\le N$ active vertices at time $0$. We start  the branching random walk $\V$ with  $N_0$ blue particles,  each particle corresponding to an active vertex in $Z_0$.  We let $\mathcal Q_0$ be the collection $\{\{u\},\, u\in Z_0\}$.  We denote by $\#G_0$ the number of vertices in $G_0$. Let $m:=\#G_0- \#Z_0$.   Each $u\in Z_0$ corresponds to a blue particle in $\V$, still denoted by $u$. We have $\V(u)=V(u)=0$. We  choose $m$ children $v$ of $u$ in $\V$ uniformly among the $n-1$ leftmost ones and color them {\it blue.}  By the coupling described above, for each blue particle $u$, $\V(v)-\V(u)$ is coupled with $X(u, v)$ for $v \in G_0 \backslash Z_0$ (where, as for $u$, the variable $v$ denotes both a particle in $\V$ and a vertex in $G_0$).  Let $i_v$ be the index of the point $T_i$ corresponding to $v$ in the coupling \eqref{coupling}. We have  $$X(u,v) = X_{(i_v)} \le \frac{n}{n-i_v} e T_{i_v} -1 =   \frac{n}{n-i_v} (\V(v)-\V(u))+\frac{i_v}{n-i_v}.$$ We also have $X(u,v) \ge \V(v)-\V(u)$. %where $i_v$ is the index of the point $T_i$ corresponding to $v$ in the coupling. 

We will see that $i_v \le N$ for all $v$ of interest.  We color the remaining children of $u$ in $\V$ {\it red.}  We then retain the $N$ leftmost particles of $\V$ in generation $1$ which is in fact the population of $\V^N$ at generation $1$.  This population may contain red and blue particles.  If a blue particle $v$  is selected in $\V^N$, we select the corresponding edge $(u, v)$ in $G_0$. Note that in this case we indeed have $i_v \le N$. We also remark that several selected edges may point to the same vertex $v$. If this happens, we keep blue the corresponding particle in $\V^N$ which has the minimal position in $\V$,  and color the other particles corresponding to $v$ {\it purple.} We then define $Z_1$ to be the set of blue particles in ${\mathcal V}^N$ at generation $1$, after removing all purple particles. We identify $Z_1$ with a subset of $G_0$: a blue particle of $Z_1$ in $\V^N$ corresponds exactly to one (active) vertex of $G_0$, and the paths $(u, v)_{u\in Z_0, v \in Z_1}$ form  ${\mathcal Q}_1$. Since $i_v\le N$ if $(u,v)\in {\mathcal Q}_1$,  we have $\V^N(v)-\V^N(u)= \V(v)-\V(u)$ and 
 \begin{equation}\label{eq:index}
  V(v)-V(u):=X(u,v) \le \frac{n}{n-N} (\V^N(v)-\V^N(u)) + \frac{N}{n-N}.
\end{equation} 
 Finally,  $G_1$ is the sub-complete graph formed by  removing $Z_0$ and all incident edges in $G_0$.

For the next steps, we proceed by induction on $j$ as long as $Z_j\neq \emptyset$. See Figure \ref{f:coupling}.

 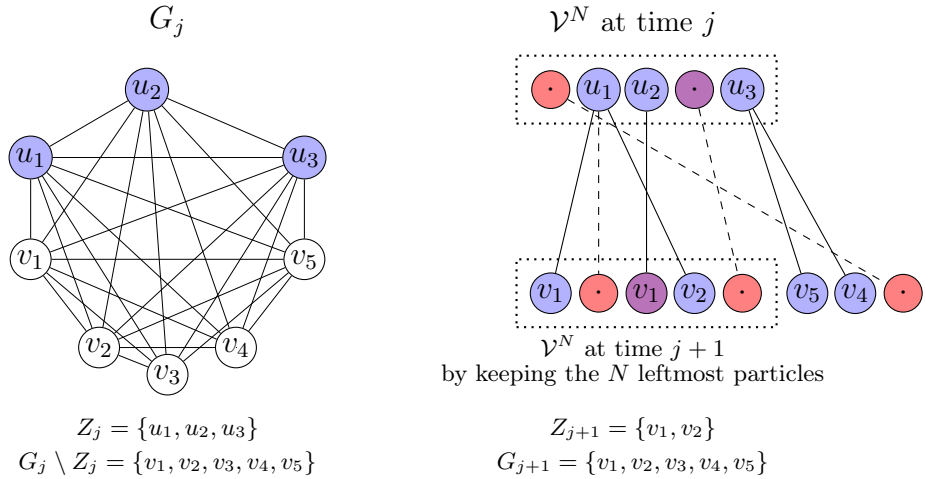
\begin{figure}[ht!]
\centering
 \begin{tikzpicture}[scale=0.9, 
    >=Stealth,
    vertex/.style={circle,draw,inner sep=1pt,minimum size=5mm},
    bluev/.style={vertex,fill=blue!30},
    redv/.style={vertex,fill=red!50},
    purplev/.style={vertex,fill=violet!55},
    thin edge/.style={line width=0.05pt},
    %thin edge/.style={ultra thin},
    dottedbox/.style={dotted,thick},
    small/.style={font=\scriptsize}
]

\node at (0,4) {$G_j$};

\node[bluev]   (u1) at (-2,2) {$u_1$};
\node[bluev]   (u2) at (-0.3,3) {$u_2$};
\node[bluev]   (u3) at (2,2) {$u_3$};

\node[vertex] (v1)  at (-2,0.5) {$v_1$};
\node[vertex] (v2)  at (-1,-0.8) {$v_2$};
\node[vertex] (v3)  at (0,-1.2) {$v_3$};
\node[vertex] (v4)  at (1,-0.8) {$v_4$};
\node[vertex] (v5)  at (2,0.5) {$v_5$};

\foreach \x/\y in {u1/u2,u2/u3,u1/u3,u1/v1,u1/v2,u1/v3,u2/v1,u2/v2,u2/v3,u3/v1,u3/v2,u3/v3, v1/v2, v1/v3, v2/v3, u1/v4, u2/v4, u3/v4, v1/v4, v2/v4, v4/v4, u1/v5, u2/v5, u3/v5, v1/v5, v2/v5, v3/v5, v4/v5}
  \draw[thin edge] (\x) -- (\y);

\node[small] at (0,-2) {$Z_j=\{u_1,u_2,u_3\}$};
\node[small] at (0,-2.5) {$G_j\setminus Z_j=\{v_1,v_2,v_3, v_4, v_5\}$};

\node at (6.8,4) {\small $\V^N$ at time $j$};

\node[bluev] (r1) at (6.3,3) {$u_1$};
\node[bluev] (r2) at (7,3) {$u_2$};
\node[bluev] (r3) at (8.4,3) {$u_3$};
\node[purplev] (r4) at (7.7,3) {$\cdot$};
\node[redv] (r5) at (5.6,3) {$\cdot$};

% Selection box V^N at time j
\draw[dottedbox] ($(r5)+(-0.5,-0.5)$) rectangle ($(r3)+(0.5,0.5)$);

%\draw[->,dashed] (u1) -- (r1);
%\draw[->,dashed] (u2) -- (r2);
%\draw[->,dashed] (u3) -- (r3);

%%\node at (6.2,1.1) {\small generation $j+1$};

% Children of u1
\node[bluev] (a1) at (5.6,0) {$v_1$};
\node[redv]  (a2) at (6.3,0) {$\cdot$};

% Children of u2
\node[purplev] (b1) at (7,0) {$v_1$};
\node[bluev]   (b2) at (7.7,0) {$v_2$};

% Children of u3
\node[redv] (c1) at (8.4,0) {$\cdot$};

%Other children
\node[bluev]  (c2) at (9.35,0) {$v_5$};
\node[bluev]  (c3) at (10.05,0) {$v_4$};
\node[redv]  (c4) at (10.75,0) {$\cdot$};

% Edges
\draw (r1)--(a1);
\draw (r1)--(b2);
\draw [dashed] (r1)--(a2);

\draw (r2)--(b1);

\draw  (r3)--(c2);
\draw  (r3)--(c3);
\draw [dashed] (r5)--(c4);

\draw [dashed](r4)--(c1);
%\draw (r1)--(c2);
%\draw (r2)--(c3);

% Selection box V^N at time j+1
\draw[dottedbox] ($(a1)+(-0.5,-0.5)$) rectangle ($(c1)+(0.5,0.5)$);

\node[small] at (6.8,-0.8) {$\mathcal V^N$ at time $j+1$};
\node[small] at (6.8,-1.2) {by keeping the $N$ leftmost particles};

% Z_{j+1}
\node[small] (Zj1) at (6.8,-2) {$Z_{j+1}=\{v_1,v_2\}$};
\node[small] (Zj1) at (6.8,-2.5) {$G_{j+1}=\{v_1, v_2, v_3, v_4, v_5\}$};
%\draw[->,dashed] (a1) -- (Zj1.west);
%\draw[->,dashed] (c1) -- (Zj1.east);

%\node[vertex] (g1) at (8.0,2.5) {$v_1$};
%\node[vertex] (g2) at (10.0,2.5) {$v_3$};
%\node[vertex] (g3) at (9,1) {$v_2$};
%
%% Edges of G_{j+1}
%\draw (g1)--(g2);
%\draw (g1)--(g3);
%\draw (g2)--(g3);
%
%\node[small] at (9,0.2) {remove $Z_j$ and incident edges};
%
\end{tikzpicture}
\caption{\small  Schematic illustration of the coupling between $(Z_j, G_j)$ and $\V^N$.}
\label{f:coupling}
\end{figure}

 Suppose we are at step $j$ and have constructed the graph $G_j$, the set $Z_j\subset G_j$ and the paths in $\mathcal Q_j$.  Each active vertex  $u\in Z_j$ corresponds to a blue particle $u$ of $\V^N$ at time $j$. We randomly color $\#G_j-\#Z_j$ children of $u$ in $\V$ blue, among the $n-1$ leftmost ones, so that each blue child $v$ corresponds to an edge in $\{u\}\times (G_j\setminus Z_j)$.  The remaining children of $u$ in $\V$ are colored {\it red}.  We couple the family $(X(u,v))_{v\in G_j\setminus Z_j}$ with the displacements of the blue particles in $\V$ via the same coupling described above.

 Purple and red particles at time $j$ give birth only to red particles. We then retain the $N$ leftmost particles of the population at generation $j+1$, which form the population of $\V^N$ at generation $j+1$.  It is composed of blue/red particles. If a blue particle $|v|=j+1$ with parent $u$ is selected in $\V^N$, we select the corresponding edge $(u,v)$ in $G_j$.  In the case several selected edges  point to the same vertex $v$ in $G_j \setminus Z_j$, so that multiple blue particles at time $j+1$ correspond to the same vertex $v$, we keep the one with minimal position in $\V$ {\it blue} and color the other particles {\it purple.}  By definition, $V(v):=V(u)+X(u,v)$ is the weight $S_{j+1}(\gamma,v)$ of the path $(\gamma,v)$ in $K_n$ obtained by  adding the edge $e=(u,v)$ to the unique path $\gamma \in {\mathcal Q}_j$ ending at $u$. The blue particles  $v$ in $\V^N$ at time $j+1$ form the set $Z_{j+1}$, and the paths $(\gamma,v)$ the set ${\mathcal Q}_{j+1}$.  Finally, we remove $Z_j$ and all incident edges to obtain the graph $G_{j+1}$.  We note that \eqref{eq:index} also holds for any $v \in Z_{j+1}$ (if it exists). Summing these inequalities, we obtain that \begin{equation}\label{eq:index-2}
V(v) \le \frac{n}{n-N}  \V^N(v)+\frac{N}{n-N}(j+1) , \qquad \forall\, v \in Z_{j+1}.
\end{equation}

\noindent Therefore, we have constructed an exploration process $(Z_j, {\mathcal Q}_j)_{j\ge 0}$, together with its coupling with $\V^N$,  up to the first index $j$ for which $Z_j=\emptyset$.  \hfill {\footnotesize $\Box$}

\medskip
   Let $|u|=j$ be a blue particle in $\V^N$ at time $j$ and $v$ be a child of $u$ in $\V$ among the $n-1$ leftmost children. Conditionally on the whole branching random walk $\V$ and %on $u$ being blue, 
    on the exploration process up to time $j$, the probability for $v$ to be colored  red is 
\[ %\begin{equation}\label{proba_red0}
1-\frac{ \#G_{j+1}}{n-1}= \frac{n-1-\#G_0+\#Z_0 + \#Z_1+\ldots+\#Z_j }{n-1}\le \frac{n-\#G_0+N(j+1)}{n-1}.
\]  %\end{equation}

Now, we estimate the probability that $v$ is colored purple, given it is not red. In particular, $v$ belongs to $\V^N$ at time $j+1$. To avoid confusion, write temporarily $w$ for the vertex in $K_n$ associated to the particle $v$ in $\V^N$. Whenever $v$ is colored purple, %the latter event occurs, 
 there must exist some particle $u'\neq u$  in $\V^N$ at time $j$ and some child $v'$ of $u'$ in $\V^N$ such that %$X(u', w)$ 
 $\V^N(v')-\V^N(u')$  is less than $x:=\V^N(v)- \V^N(u')$, and $v'$ is coupled with the vertex $w$. Let  ${\mathcal N}_{u'}$ be the set of children of $u'$ that are not red and  whose displacement (in $\V$) with respect to $u'$ is less than $x$.  By our coupling, we identify ${\mathcal N}_{u'}$ with a set of vertices in $G_{j+1}$. In particular, $w\in {\mathcal N}_{u'}$.
  By symmetry, conditionally on $\V$, on the exploration process up to time $j$, and on the red/non-red particles at time $j+1$,   the probability that $w\in {\mathcal N}_{u'}$ is $\frac{\#{\mathcal N}_{u'}}{\#G_{j+1}}$. %the conditional expectation of $\#{\mathcal N}_{u'}$ divided by $\#G_{j+1}$. % $\p(w \in {\mathcal N}_{u'} \, |\, \V, Z_0,\ldots,Z_j ) = \frac{\e (\#{\mathcal N}_{u'}\, |\, \V,  Z_0,\ldots,Z_j )}{\#G_{j+1}}$. 
  Since $\sum_{u'} \#{\mathcal N}_{u'} \le N$, we obtain that the conditional probability that $v$ is colored purple, given it is not colored red, is less than  
 \[
  \frac{N}{\#G_{j+1}}\le \frac{N}{\#G_0-N(j+1)}
  \]
as long as $N(j+1)<\#G_0$.  Since $\p(v \textrm{ is not blue} \mid \V) = \p(u \textrm{ is not blue} \mid \V)+\p(u \textrm{ is blue}, v \textrm{ is not blue} \mid \V)$, we find by induction that
 \begin{equation}\label{proba_red}
       \p( v \textrm{ is not blue} \mid \V) \le \sum_{k=0}^j \Big(\frac{n-\#G_0+N(k+1)}{n-1}  +\frac{N}{\#G_0-N(k+1)}\Big).
  \end{equation}

\subsection{Lower bound in Theorem \ref{t:main}}
The lower bound of Theorem \ref{t:main} is a consequence  of the following lemma.
\begin{lemma}
 Let $\alpha < \frac{\pi^2}{2}$ and $t<\frac{1}{\frac{\pi^2}{2}-\alpha}$. Choose $a\in (0,1)$ close enough to $1$ such that 
\[
- a + \frac{\pi^2}{2 a^2} t < \alpha t .
\]
Apply the above exploration process with $G_0=K_n$ and $N_0=N=\lfloor n^a\rfloor$. With probability tending to $1$ as $n\to\infty$, there is an active vertex  at time $\lfloor t\ln^3 n\rfloor $ such that $V(u)\le \alpha t\ln n$.  
\end{lemma}
 \begin{proof}
   For simplicity, suppose that $N=n^a$ and that $L:=t\ln^3 n= \frac{t}{a^3} \ln^3 N$ is an integer.  Note that $(-1+ \frac{\pi^2}2 \frac{t}{a^3}) \ln N= (-a + \frac{\pi^2}{2 a^2} t) \ln n $. Fix   $b\in (0,\alpha)$ such that $- a + \frac{\pi^2}{2 a^2} t<  b t$. By Proposition \ref{p:first_moment} (with $a=1$ there), there is  a particle $|u|=L$ in $\V^N$ with $\V^N(u)\le b t \ln n$ with probability going to $1$ as $n\to\infty$. By \eqref{proba_red}, the probability for $u$ to be blue also tends to $1$ as $n\to\infty$.  It remains to prove that $V(u)\le \alpha t \ln n$. By \eqref{eq:index-2},  $V(u)\le (bt \ln n)  \frac{n}{n-N} + \frac{N}{n-N} L \le \alpha t \ln n$ for $n$ large enough indeed. 
\end{proof}

\subsection{Lower bound in Theorem \ref{t:main2}}

Let $\tau_0=0$ and $\Delta> 1$ be an integer. We apply the exploration mechanism of  Section \ref{s:exploration} with $G_0=K_n$ and $N_0=1$. We start with a single active vertex, say $u^*_0$, at time $0$.  At each time $j\ge 0$ such that $Z_j\neq \emptyset$, we write $u^*_j$ for the active vertex  in $Z_j$ with the smallest position in $\V^N$, and put in store the edge $(u^*_j,v^*_j)$  corresponding to the smallest label among the edges of $G_j\setminus Z_j$ incident to $u^*_j$. Note that $v^*_j$ does not necessarily belong to $Z_{j+1}$. We set $\eta_{j+1}:=X(u^*_j,v^*_j)$.  We continue the exploration until time $\Delta$ or until $Z$ becomes empty:
$$ \tau_1:= \Delta\wedge \min\{i\ge 1: Z_i=\emptyset\}   .$$

Let $\Gamma_1\in {\mathcal Q}_{\tau_1-1}$  be the path of the exploration ending at $u^*_{\tau_1-1}$. We extend it by adding the edge $(u^*_{\tau_1-1},v^*_{\tau_1-1})$. Next, remove $Z_{\tau_1-1}$ and all its incident edges from $G_{\tau_1-1}$ to obtain $G_{\tau_1}$. We consider $v^*_{\tau_1-1}$ as the unique active vertex at time $\tau_1$.

 We then restart the same exploration process on the graph $G_{\tau_1}$, instead of $G_0$,  with initial active vertex $v^*_{\tau_1-1}$. Proceeding in the same way, we define successive times $\tau_1<\tau_2<\ldots$ and the corresponding path $(\Gamma_1,v^*_{\tau_1-1},\Gamma_2,v^*_{\tau_2-1},\ldots)$ until the first index $i$ such that $G_{\tau_i}=\emptyset$. Since at most $N$ vertices are removed at each step, such an index $i$ must be larger than $\frac{n}{\Delta N}$. See Figure \ref{fig:block-exploration}. Finally,  for each active vertex $u$, we let $V(u)$ be the weight of the unique path from the exploration ending at $u$. We note that $(\eta_j)$  will be useful for controlling $V(u^*_j)$; see \eqref{control-Vuj} below.

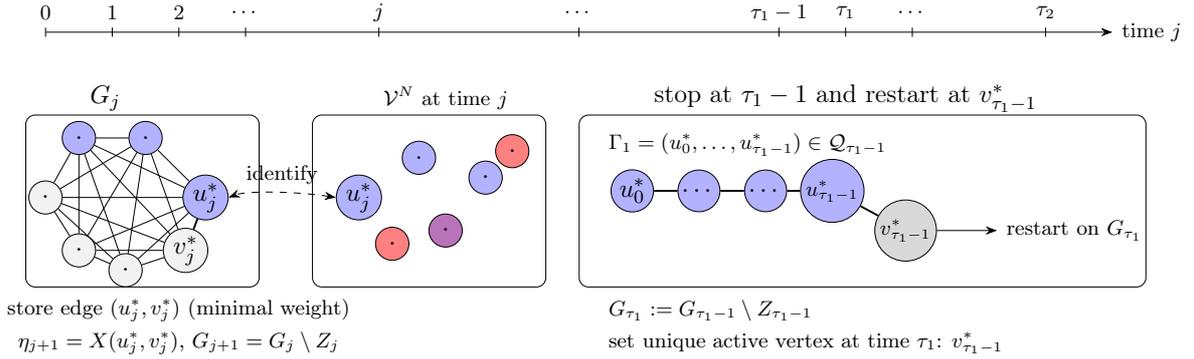
\begin{figure}[ht]
\centering
\resizebox{0.98\textwidth}{!}{%
\begin{tikzpicture}[
    >=Stealth,
    every node/.style={font=\small},
    small/.style={font=\scriptsize},
    vertex/.style={circle,draw,inner sep=1pt,minimum size=5mm},
    active/.style={vertex,fill=blue!30},
    store/.style={vertex,fill=black!15},
    inactive/.style={vertex,fill=gray!10},
     redv/.style={vertex,fill=red!50},
    purplev/.style={vertex,fill=violet!55},
    edge/.style={very thin},
    sel/.style={thick},
    dashedsel/.style={thick,dashed},
    box/.style={rounded corners,draw,very thin}
]

 %\node[small] at (1,4.5) {exploration in $\mathcal V^N$ and $G_j$};

\draw[->] (0,3.7) -- (16,3.7) node[right,small] {time $j$};

% Markers 0,...,tau1 and restart
\foreach \x/\lab in {0/0,1/1,2/2,3/{\cdots},5/{j}, 8/{\cdots}, 11/{\tau_1-1}, 12/{\tau_1},13/{\cdots}, 15/{\tau_2}}{
  \draw (\x,3.63) -- (\x,3.77);
  \node[small] at (\x,4) {$\lab$};
}
 
 \node at (0.9,2.7) { $\scriptsize G_j$};

\node[box,minimum width=3.5cm,minimum height=2.6cm,anchor=north west] (Gjbox) at (-0.3,2.45) {};

% Active vertex u_j (center)
\node[active] (ug) at (2.4,1.2) {$u^*_j$};

% --- Top row (3 vertices)
\node[active]   (t1) at (0.5,2.1) {$\cdot$};
\node[active]   (t2) at (1.5,2.10) {$\cdot$};
%\node[active]   (t3) at (1.90,2.05) {$\cdot$};

% --- Bottom row (4 vertices)
\node[inactive] (b1) at (0,1.2) {$\cdot$};   % stored endpoint
\node[inactive] (b2) at (0.5,0.40) {$\cdot$};
\node[inactive] (b3) at (1.2,0.10) {$\cdot$};
\node[inactive] (b4) at (2.1,0.4) {$v^*_j$};

% --- List of vertices
\def\GjVertices{ug,t1,t2,b1,b2,b3,b4}

% --- Complete graph: draw all edges
\foreach \i [count=\ni] in \GjVertices {
  \foreach \j [count=\nj] in \GjVertices {
    \ifnum\nj>\ni
      \draw[edge] (\i) -- (\j);
    \fi
  }
}

% Stored edge (u_j, v_j): choose b4 as v_j
%%\node[store] (vjg) at (2.00,0.55) {$v_j$};
\draw[sel] (ug) -- (b4);
\node[small] at (2,-0.5) {store edge $(u^*_j,v^*_j)$ (minimal weight)};
\node[small] at (2,-1) {$\eta_{j+1}=X(u^*_j,v^*_j)$, $G_{j+1}=G_j\setminus Z_j$};

\node at (6.0,2.7) {\scriptsize $\mathcal V^N$ at time $j$};

\node[box,minimum width=3.5cm,minimum height=2.6cm,anchor=north west] (Vjbox) at (4.0,2.45) {};

\node[active] (uV) at (4.7,1.2) {$u^*_j$};
\node[active] (uV2) at (5.6,1.8) {$\cdot$};
\node[active] (uV3) at (6.6,1.5) {$\cdot$};
\node[purplev] (uV4) at (6.0,0.7) {$\cdot$};
\node[redv] (uV5) at (7,1.9) {$\cdot$};
\node[redv] (uV6) at (5.2,0.5) {$\cdot$};

%%\draw[->,dashedsel] (uV) -- ++(-0.9,0) node[left,small] {leftmost};

% Arrow linking u_j in V^N to same vertex u_j in G_j
\draw[<->,dashed] (uV.west) .. controls (3.7,1.3) and (3.0,1.3) .. (ug.east);
\node[small] at (3.55,1.55) {\scriptsize identify};

\node at (12,2.7) {stop at $\tau_1-1$ and restart at $v^*_{\tau_1-1}$};

% Path Gamma1 and last stored edge
\node[box,minimum width=8.5cm,minimum height=2.6cm,anchor=north west] (restartbox) at (8,2.45) {};
%%\node[small,anchor=north west] at (8.2,2.42) {end of block};

\node[small,anchor=west] at (8.3,2.00)
{$\Gamma_1=(u^*_0, \dots,u^*_{\tau_1-1})\in\mathcal Q_{\tau_1-1}$};

% Draw a small chain for Gamma1
\node[active] (g0) at (8.8,1.3) {$u^*_0$};
\node[active] (g1) at (9.8,1.3) {$\cdots$};
\node[active] (g2) at (10.8,1.3) {$\cdots$};
\node[active] (g3) at (11.8,1.3) {\scriptsize $u^*_{\tau_1-1}$};

\draw[sel] (g0)--(g1);
\draw[sel] (g1)--(g2);
\draw[sel] (g2)--(g3);

% Add last stored edge (u_{tau1-1}, v_{tau1-1})
\node[store] (lastv) at (12.9,0.7) {\scriptsize $v^*_{\tau_1-1}$};
\draw[sel] (g3) -- (lastv);
%\node[small] at (11.4,0.35) {extend by $(u_{\tau_1-1},v_{\tau_1-1})$};

% Graph update: remove Z_{tau1-1} to get G_{tau1}
\node[small,anchor=west] at (8.3,-0.5)
{ $G_{\tau_1}:=G_{\tau_1-1}\setminus Z_{\tau_1-1}$};

\node[small,anchor=west] at (8.3,-1)
{set unique active vertex at time $\tau_1$: $v^*_{\tau_1-1}$};

% Restart arrow to next block
\draw[->] (lastv.east) -- ++(0.9,0) node[right,small] {restart on $G_{\tau_1}$};

\end{tikzpicture}%
}
\caption{\small  Schematic illustration of the exploration process in the proof of the lower bound of Theorem \ref{t:main2}. The path $\Gamma_1$ is the unique path in the exploring process from $u^*_0$ to $u^*_{\tau_1-1}$. Note that $u^*_j$, for $1\le j < \tau_1-1$, do not necessarily lie on $\Gamma_1$. When a particle $u$ is selected in $\V^N$, we write interchangeably $\V(u)$ or $\V^N(u)$.  } 
\label{fig:block-exploration}
\end{figure}

 The lower bound of Theorem \ref{t:main2} is a consequence  of the following lemma.   Observe that if $\beta_n\in (0,\varepsilon)$ and $\liminf_{n\to\infty} \frac{\sqrt{2\beta_n}}{\pi} \ln n > 1$, then $\liminf_{n\to\infty} \frac{\sqrt{2\widetilde{\beta}_n}}{\pi}\ln n > 1 + \limsup_{n\to\infty} \widetilde{\beta}_n^{1/4} $ where $\widetilde{\beta}_n$ is $\max( \beta_n, (1+\widetilde{\varepsilon})^{2} \frac{\pi^2}{2 \ln^2 n}  )$ or $\min( \beta_n, (1+\widetilde{\varepsilon})^{2} \frac{\pi^2}{2 \ln^2 n}  )$ with any choice $\widetilde{\varepsilon}>\varepsilon^{1/4}$. This observation implies that,    in the assumptions on $\beta_n$ in Theorem \ref{t:main2},  we may additionally require, at no cost,  that $\liminf_{n\to\infty} \frac{\sqrt{2{\beta}_n}}{\pi}\ln n > 1 + \limsup_{n\to\infty} {\beta}_n^{1/4} $.

\begin{lemma} \label{lem:low-T2}
 For each $\varepsilon$, let  $\beta_n$ be a sequence  with $\beta_n \in (0,\varepsilon)$ for all $n$ and such that $\liminf_{n\to\infty} \frac{\sqrt{2{\beta}_n}}{\pi}\ln n > 1 + \limsup_{n\to\infty} {\beta}_n^{1/4} $. Let $b>a>1$  such that $\liminf_{n\to\infty} \frac{\sqrt{2\beta_n}}{\pi} \ln n > b$ and $b>a+\limsup_{n\to\infty} \beta_n^{1/4}$. 
 Set $N=\lfloor e^{\frac{a\pi}{\sqrt{2\beta_n}}}\rfloor $ and $L=ne^{-\frac{b\pi}{\sqrt{2\beta_n}}}$. In the exploration described above, 
 choose $\Delta=\lfloor  \ln^3 N\rfloor $. With probability tending to $1$ as $n\to\infty$ then $\varepsilon\to 0$, there exists an active vertex  $u$  at a time greater than $L$ such that $V(u)\le \beta_n L$. 
\end{lemma}
\begin{proof}
   For simplicity, we treat $\ln^3 N$, $L$ and $e^{\frac{a\pi}{\sqrt{2\beta_n}}}$  as  integers. Observe that 
   \begin{equation}\label{limit_N}
    \lim_{n\to\infty} \frac{N}{n}=\lim_{n\to\infty} \frac{N\Delta}{n\beta_n}=\lim_{n\to\infty} \frac{\Delta}{L}=0.
   \end{equation}
   Moreover $\Delta\to\infty$, $  \frac{ N L}{n} \to 0$ when $n\to\infty$ and then $\varepsilon\to 0$.  Actually,  $NL \sim n e^{-\frac{b-a}{\sqrt{2\beta_n}}} $ and $\Delta \sim  (\frac{a\pi }{\sqrt{2\beta_n}})^3$. %$NL\Delta= t(\frac{a\pi }{\sqrt{2\beta_n}})^3 n e^{-\frac{b-a}{\sqrt{2\beta_n}}} $, 
Hence for any $c>0$,
\begin{equation}\label{eq:NLDelta}
\lim_{\varepsilon\downarrow 0} \limsup_{n\to\infty}  \frac{1}{n}NL\Delta^{c}=0.
\end{equation}
 We let 
\[
i_L:=\inf\{i\colon \tau_i>  L\}.
\]

\noindent Recall that $\tau_i$ is well-defined at least for $i \le \frac{n}{\Delta N}$. 
We will prove the lemma with $u=u^*_{\tau_{i_L}-1}$.  Since $L < \frac{n}{\Delta N}$ and  $\tau_i- \tau_{i-1} \ge 1$,  $i_L\le L+1$. As a matter of fact, we will show  that for any $a_1 \in (1, a)$, there is some $\varepsilon_0>0$ small enough such that for all $\varepsilon \in (0, \varepsilon_0)$,  \begin{equation}\label{bound_iL}
\lim_{n\to\infty}\p\Big(i_L \in \Big[\frac{L}{\Delta}, \frac{a_1 L}{\Delta}\Big]\Big)=1.
\end{equation}

Let $\gamma$ be the path 
\[
\gamma:=(\Gamma_1,v^*_{\tau_1-1},\Gamma_2,\ldots,v^*_{\tau_{i_L-1}-1},\Gamma_{i_L}).
\]
Note that 
\begin{equation}\label{eq:Sgamma}
V(u^*_{\tau_{i_L}-1})=S_{|\gamma|}(\gamma)=\sum_{i=1}^{i_L} {\tt s}_i  +  \sum_{i=1}^{i_L-1} \eta_{\tau_i},
\end{equation}

\noindent where ${\tt s}_i:= S_{|\Gamma_i|}(\Gamma_i) = V(u^*_{\tau_i-1})- V(v^*_{\tau_{i-1}-1})$ is the total weight of the path $\Gamma_i$ in the sense of \eqref{Sigamma}, and we recall that $\eta_{\tau_i}= X(u^*_{\tau_i-1}, v^*_{\tau_i-1})$. The main contribution to  $V(u^*_{\tau_{i_L}-1})$ comes from $\sum_{i=1}^{i_L} {\tt s}_i {\bf 1}_{{\mathcal E}_i}$, where ${\mathcal E}_i$ denotes the event that, when  running the exploration process on the complete graph $G_{\tau_{i-1}}$, the leftmost particle of the associated $N$-branching random walk $\V^N$ at time $\Delta-1$  is {\it blue}. Indeed,  we will show that for any $\delta>0$, 
\begin{align} & \p\Big( \sum_{i=1}^{i_L-1} \eta_{\tau_i} > \delta \beta_n L\Big) = o_{\varepsilon,n}(1), \label{sum-eta-tau}\\
  & \p\Big( \sum_{i=1}^{i_L}   ({\tt s}_i+\Delta)   {\bf 1}_{{\mathcal E}_i^c}> \delta \beta_n L\Big) = o_{\varepsilon,n}(1), \label{sum-S-tau<Delta}  \end{align}

\noindent where $o_{\varepsilon, n}(1)$ denotes some quantity which tends to $0$ as $n\to\infty$ then $\varepsilon\to0$, and the term $\Delta$ in \eqref{sum-S-tau<Delta}  comes from the forthcoming \eqref{onEi-3}.  

We postpone, for the moment,  the proofs of \eqref{bound_iL}, \eqref{sum-eta-tau} and \eqref{sum-S-tau<Delta}, and turn to   the main term  $\sum_{i=1}^{i_L} {\tt s}_i {\bf 1}_{{\mathcal E}_i}$.

 On ${\mathcal E}_i$, $\tau_i-\tau_{i-1}=\Delta$.  By \eqref{eq:index-2}, we can find $(\Theta_i)_{i\ge 1}$, a sequence of independent copies of $\Theta:=\min_{|u|=\Delta-1} \V^N(u)$, such that for any $i$,    \begin{equation} {\tt s}_i    \le \frac{n}{n- N} \Theta_i + \frac{N}{n- N} \Delta
= \Theta_i+  \frac{N}{n- N}  (\Delta+ \Theta_i), \qquad \mbox{on ${\mathcal E}_i$}.  \label{onEi}\end{equation}

We  mention that $\min_{|u|=\Delta-1} \V^N(u) \ge -\Delta$, hence $\Theta_i \ge -\Delta$. This in view of \eqref{onEi} implies that for any $i $,   \begin{align} {\tt s}_i  {\bf 1}_{{\mathcal E}_i}  & \le \Theta_i+  \frac{N}{n- N}  (\Delta+  \Theta_i) - \Theta_i {\bf 1}_{{\mathcal E}^c_i} 
\nonumber\\
&\le
\Theta_i+  \frac{N}{n- N}  (\Delta+ \Theta_i) +\Delta {\bf 1}_{{\mathcal E}^c_i} .   \nonumber\end{align} 

\noindent  Therefore $$ \sum_{i=1}^{i_L} {\tt s}_i  {\bf 1}_{{\mathcal E}_i}  
\le \frac{n}{n- N}  \sum_{i=1}^{i_L}\Theta_i + \frac{N \Delta }{n- N}   i_L  +
\sum_{i=1}^{i_L}\Delta {\bf 1}_{{\mathcal E}^c_i} .
$$
Since $i_L\le L+1$, we finally obtain that  \begin{equation}  \sum_{i=1}^{i_L} {\tt s}_i   
\le \frac{n}{n- N}  \sum_{i=1}^{i_L}\Theta_i + \frac{N \Delta (L+1)}{n- N}    +
\sum_{i=1}^{i_L}({\tt s}_i+\Delta) {\bf 1}_{{\mathcal E}^c_i} .
\label{onEi-3}\end{equation}

Recall that $a_1>1$. By Proposition \ref{p:first_moment} (with $a=0$ there), for all large $N$, $$ \e (\Theta_1) \le a_1 \frac{\pi^2}{2} \ln N.$$ 

Choose $a_2\in (a_1, a)$. 
In view of \eqref{bound_iL},   we deduce from Corollary \ref{coro:concentration} that  with probability at least $1- o_{\varepsilon, n}(1)$, $$\sum_{i=1}^{i_L}\Theta_i \le \frac{a_2L}{\Delta} \times a_1 \frac{\pi^2}{2} \ln N \sim \frac{a_1 a_2}{a^2} \beta_n L.$$

\noindent By \eqref{limit_N}, $\frac{n}{n- N}  \to 1$  and 
 for any $\delta>0$, $\frac{N \Delta (L+1)}{n- N}< \delta \beta_n L$ for all $n$ large enough.  Choose $\delta \in (0, (1-\frac{a_1 a_2}{a^2})/4)$.  We deduce from  \eqref{sum-eta-tau} and \eqref{sum-S-tau<Delta} that $$\p\Big(V(u^*_{\tau_{i_L}-1})> \beta_n L\Big) = o_{\varepsilon, n}(1),$$ proving the lemma. 

\medskip
It remains to prove \eqref{bound_iL}, \eqref{sum-eta-tau} and \eqref{sum-S-tau<Delta}.

We start with the proof of  \eqref{bound_iL}. Recall that  ${\mathcal E}_1$ is the event that the leftmost particle of $\V^N$ at time $\Delta-1$ is  blue. For any  complete graph $G$, we denote by $\p_{G}$ the probability relating to the exploration of $G$ starting with one active vertex.   By \eqref{proba_red} with $G_{\tau_i}$ in place of $G_0$, for all $i\ge 0$ such that $\#G_{\tau_i}> N \Delta$,
\[
\p_{G_{\tau_i}}(\mathcal E_1^c) \le \Big(\frac{n-\#G_{\tau_i}+N\Delta}{n-1}+\frac{N}{\#G_{\tau_i}-N\Delta}\Big)\Delta .
\]

For all $n$ large enough, we have   $\Delta < L$.  Moreover, observe that at each step of the exploration,  we delete  at most $N$ vertices from the graph. Thus we have  that for any $1\le j \le \tau_{i_L-1}$, 
\begin{equation}\label{card_G}
\#G_{j} =n- (\#Z_0+\ldots+\#Z_{j-1}) \ge n- Nj\ge n- NL .
\end{equation}

\noindent  We get that, for all $i< i_L$,   \begin{equation}\label{bound_eventE}
\p_{G_{\tau_i}} (\mathcal E_1^c) \le  (\frac{2L}{n-1} +\frac{1}{n-2NL}) N\Delta.
\end{equation}
%We used that $\tau_i+\Delta \le 2L$. 
If $\tau_1<\Delta$, then $Z_{\Delta-1}=\emptyset$, hence $\mathcal E_1^c$ occurs. Therefore
\begin{equation}\label{bound_tau1}
\p_{G_{\tau_i}} (\tau_1< \Delta) \le (\frac{2L}{n-1} +\frac{1}{n-2NL}) N\Delta .
\end{equation}

Let $p \in (\frac1{a_1}, 1)$.  Recall the asymptotics given in \eqref{eq:NLDelta} and the paragraph preceding it. We deduce that if $\varepsilon$ is chosen small enough, for all $n$ large enough and all $i<i_L$,
\[
\p_{G_{\tau_i}} (\tau_1=  \Delta) \ge  p.
\]

\noindent 
Note that conditionally on $\sigma\{G_j, j\le \tau_i\}$, $\tau_{i+1}-\tau_i$ is distributed as $\tau_1$ under $\p_{G_{\tau_i}} $. Then for any $j\ge 1$, $\#\{i \in \lb 0,  j-1\rb: \tau_{i+1}-\tau_i=\Delta\}$ is stochastically larger than the sum of $j$  independent Bernoulli r.v. with parameter $p$.  By the weak law of large numbers, with probability larger than $1- o(1)$ with some term $o(1) \to 0$ as $n\to\infty$, there are more than $\frac{L}{\Delta}$ indices $i\le \frac{a_1 L}{\Delta}$ such that $\tau_{i+1}= \tau_i+\Delta$. On this event,  $i_L\le \frac{a_1 L}{\Delta}$.   This implies the upper bound in \eqref{bound_iL}. The lower bound $i_L\ge \frac{L}{\Delta}$ is trivial as $\tau_i- \tau_{i-1} \le \Delta$. 

\medskip
We turn to the proof of \eqref{sum-eta-tau}. Observe that conditionally on the exploration up to time $j$,  the variable $\eta_j$ is the minimum of $\#G_{j+1}$ independent r.v. distributed as $X$. As long as $\#G_{j+1}\ge \frac{n}2$, $\eta_j$ is stochastically smaller than an exponential random variable with mean $2e$. We can assume that the number of vertices in $G_{\tau_{i}}$ is greater than $\frac{n}{2} $ for any $i<i_L$ and that $i_L\le \frac{a_1 L}{\Delta}$  by \eqref{card_G} and \eqref{bound_iL}. Note that on $\{i_L\le \frac{a_1 L}{\Delta}\}$, $\sum_{i=1}^{i_L-1} \eta_{\tau_i}$ is bounded above by the sum of the $\lfloor \frac{a_1 L}{\Delta}\rfloor$ largest terms among $\eta_0, ..., \eta_{L-1}$. Applying Lemma \ref{l:upper_stoc} to $k=\lfloor \frac{a_1 L}{\Delta}\rfloor$ and $\ell=L$ gives that (recalling that $a_2 > a_1$) $$\e\Big[\Big(\sum_{i=1}^{i_L-1} \eta_{\tau_i} \Big) {\bf 1}_{\{i_L\le \frac{a_1 L}{\Delta}\}}\Big] \le \frac{2 e a_2 L}{\Delta} \ln \Delta.$$

Since $\frac{ \ln \Delta}{\Delta} =o(\beta_n)$,  the Markov inequality and \eqref{bound_iL} yield \eqref{sum-eta-tau}.

\medskip
Finally,  we prove \eqref{sum-S-tau<Delta}.  By \eqref{bound_iL} and the positivity of ${\tt s}_i+\Delta$, it is enough to show that \begin{equation}    \p\Big( \sum_{i=1}^{a_1L/\Delta}   ({\tt s}_i+\Delta)   {\bf 1}_{{\mathcal E}_i^c}> \delta \beta_n L\Big) = o_{\varepsilon,n}(1). \label{sum-S-tau<Delta-2} \end{equation}

 Note that \begin{equation}    \label{cond-Gi}
\e[({\tt s}_i+\Delta)   {\bf 1}_{{\mathcal E}_i^c}]= \e\Big[  \e_{G_{\tau_{i-1}}}[ ({\tt s}_1+\Delta) {\bf 1}_{{\mathcal E^c_1}}]\Big].
\end{equation}

\noindent  We have ${\tt s}_1=V(u^*_{\tau_{1}-1})$. On the event $\{\tau_1>j\}$,    $\V(u^*_{j})=\V^N(u^*_j)$ being the minimum of $\V^N$ at time $j$ over active vertices, we must have $\V(u^*_{j}) \le \V(v^*_{j-1})$. In the coupling \eqref{coupling}, $\V(v^*_{j-1})-\V(u^*_{j-1}) \le X(u^*_{j-1},v^*_{j-1})=\eta_{j}$. Hence $\V^N(u^*_{j})- \V^N(u^*_{j-1}) \le \eta_{j}$ so we get $\V^N(u^*_{j})\le \sum_{i=1}^{j} \eta_i$, and by \eqref{eq:index-2},  \begin{equation}    V(u^*_{j}) \le \frac{n}{n-N} \sum_{i=1}^{j} \eta_i + \frac{Nj}{n-N}, \qquad \forall\, j< \tau_1.\label{control-Vuj} \end{equation}  

\noindent Hence $${\tt s}_1 \le \frac{n}{n-N} \sum_{i=1}^{\Delta-1} \max(\eta_i,0) + \frac{N (\Delta-1) }{n-N}.$$

 The term $\sum_{i=1}^{\Delta-1} \max(\eta_i,0)$ is stochastically smaller than the sum of $\Delta-1$ independent exponential r.v. with mean $2e$, as long as $G_{\Delta}$ has more than $\frac{n}{2}$ vertices. This condition is indeed satisfied thanks to \eqref{card_G}. Then for some positive constant $c$,  $$  \e_{G_{\tau_{i-1}}}[({\tt s}_1+\Delta)^2] \le c \Delta^2.$$ 

By Cauchy--Schwarz inequality and  \eqref{bound_eventE}, we obtain that $$  \e_{G_{\tau_{i-1}}}[ ({\tt s}_1+\Delta) {\bf 1}_{{\mathcal E^c_1}}]
\le
c^{1/2} \Delta \Big( (\frac{2L}{n-1} +\frac{1}{n-2NL})N\Delta\Big)^{1/2}
\le \Delta^{-4}, 
$$

\noindent where the last inequality holds for all small $\varepsilon$ and large $n$ (see \eqref{eq:NLDelta}). It follows from \eqref{cond-Gi} that $$ \e \Big[\sum_{i=1}^{a_1L/\Delta}   ({\tt s}_i+\Delta)   {\bf 1}_{{\mathcal E}_i^c}\Big]
\le a_1 L \Delta^{-3}.
$$

\noindent Since $L \Delta^{-3}= o(\beta_n L)$, the Markov inequality yields \eqref{sum-S-tau<Delta-2} 
and completes the proof of the lemma.
\end{proof}

\section{Proofs of Lemmas \ref{main:upper}, \ref{main2:upper},  and of Proposition  \ref{p:first_moment}}\label{s:rw-brw}

Recall that $(Y_n)_{n\ge 0}$ is a centered random walk starting at $0$ with step distribution $\mbox{Exp}(1)-1$. In the first subsection, we present several estimates on $(Y_n)_{n\ge 0}$ and provide  the proofs of Lemmas \ref{main:upper}, \ref{main2:upper}. The second subsection is devoted to the study of the branching random walks $\V$ and $\V^N$, and to the proof of Proposition  \ref{p:first_moment}. 

\subsection{Estimates of random walks: Proofs of Lemmas \ref{main:upper} and \ref{main2:upper}}\label{sub:rw}

\begin{lemma}\label{l:mogulski}
 Let $\upsilon>0$ and $0<d<\frac{\pi^2}{2}$. There exists $\Delta_0>0$ such that for all $\Delta\ge \Delta_0$,   $r_0\le 0\le r_0+\Delta$ and  $\ell\ge 1$,
\begin{equation}
    \p(r_0  \le Y_j\le r_0+\Delta,\, \forall \, j\in \lb 1,  \ell\rb) \le e^{\Delta^\upsilon-\frac{d}{\Delta^2}\ell}. \label{mogulskii_upper0} 
\end{equation}
\end{lemma}
\begin{proof}
 By  \cite[Theorem 1]{mogulski}, for all $a\in (0,1)$, and $L_{\Delta}$ such that $\lim_{\Delta\to\infty}\frac{L_{\Delta}}{\Delta^2}=\infty$,
\[
    \lim_{\Delta\to\infty} \frac{\Delta^2}{L_\Delta}\ln \p(-a\Delta  \le Y_j\le (1-a)\Delta,\, \forall \,  j\in \lb 1,  L_{\Delta}\rb)=-\frac{\pi^2}{2}.
\]

 Fix $m\ge 1$ such that $d (1+ \frac1{m})^2 < \frac{\pi^2}{2}$.
Let $L=\lfloor\frac{1}{d} \Delta^{2+\upsilon} \rfloor$. The above equation applied to $(1+\frac{1}{m})\Delta$  implies that we can find $\Delta_0>0$ such that for all $\Delta\ge \Delta_0$ and all $r_0\in\{ (-1+\frac{k}{m})\Delta, \, k=0,\ldots, m-1 \}$, \[
     \p\Big(r_0  \le Y_j\le r_0+ (1+\frac{1}{m})\Delta,\, \forall \,  j\in \lb 1,  L\rb\Big) \le e^{-\frac{d}{\Delta^2} L}.
\]
If $r_0\in [  (- 1+ \frac{k}{m}) \Delta, (- 1+ \frac{k+1}{m})\Delta]$ for some $k=0,\ldots,m-1$, we have
\[
\p\Big(r_0  \le Y_j\le r_0+\Delta,\, \forall \, j\in \lb 1,  L\rb)\le 
\p\Big( (- 1+ \frac{k}{m}) \Delta \le Y_j\le  \frac{k+1}{m}\Delta,\, \forall \, j\in \lb 1,  L\rb\Big).
\]
We have proved that for all $r_0\le 0\le r_0+\Delta$,
\[
\p(r_0  \le Y_j\le r_0+\Delta,\, \forall \, j\in \lb 1,  L\rb) \le e^{-\frac{d}{\Delta^2}L}.
\]
To finish the proof, it suffices to split any interval $\lb 0,\ell\rb$ into intervals of length $L$ and use the Markov property. For the bit of interval which has length smaller than $L$,  we use that $\Delta^{\upsilon} \ge \frac{d}{\Delta^2} L$. 
\end{proof}

We extend Lemma \ref{l:mogulski} to include the case of a small drift:

\begin{lemma}\label{l:upper} 
 Let $\upsilon>0$, $0< d<\frac{\pi^2}{2}$ and $c\ge 0$. There exists $\Delta_0>0$  such that for all $\Delta\ge \Delta_0$,   $r_0\le 0\le r_0+\Delta$ and  $\ell\ge 1$,
\begin{equation}
   \p(r_j  \le Y_j\le s_j,\, \forall \, j\in \lb 1,  L\rb) \le e^{\Delta^\upsilon-\frac{d}{\Delta^2}\ell}, \label{mogulskii_upper} 
\end{equation}
where $r_j=r_0+\frac{c}{\Delta^2}j$, $s_j=r_j+\Delta$. 
\end{lemma}

\begin{proof}
Let $a=\frac{c}{\Delta^2}$. For $q<1$, we have $\e[e^{qY}]=\frac{e^{-q}}{1-q}$ and $\frac{\e[Ye^{qY}]}{\e[e^{qY}]} = \frac{q}{1-q}$. We choose $q=\frac{a}{1+a}$ so that $Y$ has mean $a=\frac{c}{\Delta^2}$ after tilting. The tilted distribution is then the distribution of $a + (1+a)Y$.  Using this change of measure and the identities $\frac{1}{1-q}=1+a$, $q\ell +q(a\ell +(1+a)Y_\ell)=a(Y_\ell+\ell)$, we get 
\begin{align*}    
    \p(r_j  \le Y_j\le s_j,\, \forall \, j\in \lb 1,  L\rb)
    &=
     (1+a)^{\ell} \e\Big[e^{- a(Y_\ell+\ell)} {\bf 1}_{\{ \frac{r_0}{1+a}  \le Y_j\le \frac{r_0+\Delta}{1+a},\, \forall \, j\in \lb 1,  L\rb \}}\Big]
     \\
     &\le
     e^{-  \frac{a r_0}{1+a}} \p \Big(\frac{r_0}{1+a}  \le Y_j\le \frac{r_0+\Delta}{1+a},\, \forall \, j\in \lb 1,  L\rb\Big).
  \end{align*}
      Since $|a r_0| \le \frac{c}{\Delta}$,   the conclusion follows from    \eqref{mogulskii_upper0}. 
\end{proof}

\begin{lemma}\label{l:mogulskii_lower}
 Let $\upsilon >0 $, $d'>\frac{\pi^2}{2}$, $c\ge 0$ and $\varepsilon>0$. There exists $\Delta_0>0$ such that for $\Delta\ge \Delta_0$,  $r_0\le 0\le r_0+\Delta$ and  $\ell\ge \lfloor \Delta^{2+\varepsilon}\rfloor$,
\begin{equation}
    \p(r_j  \le Y_j\le s_j,\, \forall \,   j\in \lb 1, \ell\rb,\,\, Y_\ell\in [r_\ell, r_\ell+\upsilon \Delta]) \ge 
   e^{-\frac{d'}{\Delta^2}\ell}, \label{mogulskii_lower}
\end{equation}
where $r_j=r_0+\frac{c}{\Delta^2}j$, $s_j=r_j+\Delta$. Moreover, the conclusion in \eqref{mogulskii_lower} holds with $[s_\ell-\upsilon \Delta, s_\ell]$  in place of  $[r_\ell, r_\ell+\upsilon \Delta]$.
\end{lemma}
\begin{proof}   
First, we treat the case $c=0$.  
  By  \cite[Theorem 1]{mogulski}, for all $a\in (0,1)$, and $L_{\Delta}$ such that $\lim_{\Delta\to\infty} \frac{L_{\Delta}}{\Delta^2}=\infty$,
\begin{equation}\label{temp_mogulski}
    \lim_{\Delta\to\infty} \frac{\Delta^2}{L_\Delta}\ln \p(-a\Delta  \le Y_j\le (1-a)\Delta,\, \forall \,  j\in \lb 1,  L_{\Delta}\rb,\, Y_{L_{\Delta}} \in [-a\Delta, (-a+\upsilon)\Delta])=-\frac{\pi^2}{2}.
\end{equation}
Let $d''\in (\frac{\pi^2}{2},d')$ and $L=\lfloor  \Delta^{2+\frac{\varepsilon}{2}}\rfloor-\lfloor \Delta^2\rfloor$. Following a reasoning  analogous to the proof of Lemma \ref{l:mogulski}, fix $m> \max(\frac{1}{\upsilon},4)$ such that $d''(1-\frac{1}{m})^2>\frac{\pi^2}{2}$. We can find $\Delta_0>0$ such that for all $\Delta\ge\Delta_0$ and all $r_0\in \{(-\frac{2}{3}+\frac{k+1}{m})\Delta,\, k=0,\ldots,\lfloor \frac{m}{3}\rfloor\}$,
\[
\p(r_0\le Y_j\le r_0+(1-\frac{1}{m})\Delta,\, \forall \,   j\in \lb 1,  L\rb,\, Y_{L} \in [r_0,r_0+(\upsilon-\frac{1}{m})\Delta] )\ge e^{-\frac{d''}{\Delta^2}L}.
\]
If $r_0\in [(-\frac{2}{3} + \frac{k}{m})\Delta, (-\frac23 +\frac{k+1}{m})\Delta]$, we write 
\begin{align*}
& \p(r_0\le Y_j\le r_0+\Delta,\, \forall \,   j\in \lb 1,  L\rb,\, Y_{L} \in [r_0,r_0+\upsilon\Delta] ) 
\\
\ge & \p( a_{k+1} \le Y_j\le a_{k}+\Delta,\, \forall \,   j\in \lb 1,  L\rb,\, Y_{L} \in [a_{k+1},a_k+\upsilon\Delta] )
\end{align*}
with $a_j=(-\frac23 +\frac{j}{m})\Delta$. Therefore for any $r_0 \in [ -\frac23 \Delta, -\frac13 \Delta]$, 
\[
\p(r_0\le Y_j\le r_0+\Delta,\, \forall \,   j\in \lb 1,  L\rb,\, Y_{L} \in [r_0,r_0+\upsilon\Delta] )\ge e^{-\frac{d''}{\Delta^2}L}.
\]
 Moreover,  the probability for $Y$ to lie in $r_0+[\frac13 \Delta,\frac23 \Delta]$  at  time   $\lfloor \Delta^2\rfloor$ before exiting $[r_0,r_0+\Delta]$ is greater than  $\frac{C}{\Delta}\ge e^{-\frac{d'-d''}{\Delta^2}L}$ uniformly in  $r_0\in [-\Delta,0]$ and $\Delta$ large enough. By the Markov property at time $\lfloor \Delta^2\rfloor$, we get that \eqref{mogulskii_lower} holds for $r_0 \in [-\Delta,0]$ and $\ell= \ell_0:=\lfloor  \Delta^{2+\frac{\varepsilon}{2}}\rfloor$. 
For $\ell\ge  \lfloor  \Delta^{2+\varepsilon}\rfloor$, we  divide the time interval $\lb 0,\ell\rb$ into intervals of length $\ell_0$. If the first interval has length $L_0 \le \ell_0$, we apply the bound
\[
\p(r_0\le Y_j\le r_0+\Delta,\, \forall \,   j\in \lb 1,  L_0\rb)\ge \p(r_0\le Y_j\le r_0+\Delta,\, \forall \,   j\in \lb 1,  \ell_0\rb).
\]
We get that for all $\ell\ge  \lfloor \Delta^{2+\varepsilon}\rfloor$,
\[
\p(r_0\le Y_j\le r_0+\Delta,\, \forall \,   j\in \lb 1, \ell\rb,\, Y_\ell \in [r_0,r_0+\upsilon\Delta])\ge e^{-\frac{d'}{\Delta^2}(\ell+\ell_0)}.
\]
It remains to observe that $\ell_0 = o( \Delta^{2+\varepsilon})$ as $\Delta\to \infty$. Thus, we have proved \eqref{mogulskii_lower} for $c=0$. The general case is deduced by a change of measure as in Lemma \ref{l:upper}. 

Finally, the same conclusion holds with $[s_\ell-\upsilon \Delta, s_\ell]$ instead of $[r_\ell, r_\ell+\upsilon \Delta]$. Its proof being entirely analogous,  we omit the details. \end{proof}

We can now give the proof  of Lemma \ref{main:upper}.

\begin{proof}[Proof of Lemma \ref{main:upper}.]  It suffices to prove the lemma for $c\in (0,\frac12)$ and $L\ge \Delta^{\frac{2}{1-2c}}$. Indeed, for $c>0$, we can choose $c'\in (0,\frac{1}{2})$ such that $c'<c$ and $\frac{2}{1-2c'}<2+c$.

For $d \in (0, \frac{\pi^2}2)$, let  $\eta \in (0,1)$ be sufficiently small such that $(2 + \eta) (\frac12- c) < 1$ and $d (1+\eta)^4 < \frac{\pi^2}2$. Consider large $\Delta$  (how large will be made clear  later) and $L\ge \Delta^{\frac{2}{1-2c}}$. Notice that $\frac{L}{\Delta^{2+\eta}}\to\infty$ as $\Delta\to \infty$. We will divide the time interval $\lb 0, L\rb$ into  smaller intervals of length $  \Delta^{2+\eta}$, and similarly divide  the space into intervals of length $\eta \Delta$.  Assume for brevity that both $ \Delta^{2+\eta}$ and $\frac{L}{\Delta^{2+\eta}}$ are integers.  For $j\ge 0$, let $I_j:=\lb j \Delta^{2+\eta},  (j+1) \Delta^{2+\eta}\rb$.
For each $j$, we record $M_j$ the number of the form $\eta k \Delta $, with $k$ integer, such that $M_j\ge \max_{0\le \ell \le   (j+1) \Delta^{2+\eta}} Y_\ell > M_j - \eta \Delta$. Write $B:= \{Y_L \le \frac\Delta{c}     ,\, \min_{0\le i  \le j \le L}(Y_j-Y_i) \ge - \Delta\}$. Notice that on $B$, 
\[
\max_{\lb 0,L\rb} Y\le \frac{\Delta}{c}+ \Delta .
\] 

Let $C:= \frac1{c}+2$, and consider $j\in \lb 0,  \frac{L}{\Delta^{2+\eta}}-1\rb$.  Notice that  $M_j$ takes fewer than $\frac{C}{\eta}$ possible values. In particular, since $M_j$ is non-decreasing in $j$,  the number of indices $j$ such that $M_j\neq M_{j+1}$ (call this set $J$)  is at most  $  \frac{C}{\eta}$.

 For $j\not\in J$,   $M_j=M_{j+1}$. Furthermore, since $\max_{0\le \ell \le   (j+1)\Delta^{2+\eta}} Y_\ell  > M_j- \eta \Delta$, and on $B$, one cannot make a drop larger than $\Delta$,  we conclude that  for all $i\in I_{j+1}$, $Y_i \in [M_j- (1+\eta)\Delta, M_j]$.  Let $J_0$ be a deterministic subset of $\lb 0,\frac{L}{\Delta^{2+\eta}}-1\rb$ with cardinality smaller than $\frac{C}{\eta}$. Applying the Markov property at time $ (j+1)\Delta^{2+\eta}$ for each $j \not\in J$ (in decreasing order of $j$) and ignoring those  $j\in J$, we obtain that $$ \p\Big( B ,  \, J=J_0\Big)
 \le
\Big( \max_{m= \eta k \Delta, 0\le k \le \frac{C}\eta} \,\sup_{r\in [m-(1+\eta)\Delta,  m]} \p_r\big( m-(1+\eta)\Delta \le Y_i \le m, \forall i \le  \Delta^{2+\eta}\big)\Big)^{\#J_0^c} ,
$$

\noindent with $\#J_0^c= \frac{L}{\Delta^{2+\eta}}- \#J_0 \ge \frac{L}{\Delta^{2+\eta}} - \frac{C}{\eta}$. Applying \eqref{mogulskii_upper0} to $\upsilon=\frac\eta2$ and $ d (1+\eta)^4< \frac{\pi^2}{2}$, for all large $\Delta$, uniformly in $m$,  $$\sup_{r\in [m-(1+\eta) \Delta,  m]} \p_r\big( m-(1+\eta) \Delta \le Y_i \le m, \,\forall \,i \le \Delta^{2+\eta}\big) \le e^{\Delta^\upsilon- d   (1+\eta)^2  \Delta^\eta}  \le e^{ - d   (1+\eta)   \Delta^\eta},$$

 \noindent for all $\Delta\ge \Delta_0$ (if $\Delta_0$ is sufficiently large). Hence $$ \p\Big( B ,  \, J=J_0\Big)
 \le e^{- d(1+\eta) \frac{L}{\Delta^2} + d(1+\eta) \frac{C}{\eta} \Delta^\eta}
 \le e^{- d(1+\frac{\eta}2) \frac{L}{\Delta^2}} 
$$
if $\Delta_0$ is big enough. Summing over all  possible $J_0$ (there are at most $2^{L /\Delta^{2+\eta}}$ of them) completes the proof of the lemma, as long as $\Delta_0$ is chosen large enough.  \end{proof}

The proof of Lemma \ref{main2:upper} proceeds in a similar way as that for Lemma \ref{main:upper}: 

\begin{proof}[Proof of Lemma \ref{main2:upper}]   Let $d\in (0, \frac{\pi^2}2)$, and choose $\eta \in (0, c/2)$ such that $d (1+\eta)^4 < \frac{\pi^2}{2}$. For brevity, we assume that   $\Delta^{2+\eta}$ and $\frac{L}{\Delta^{2+\eta}}$ are integers.  We divide the space into intervals of length $\eta \Delta$, and the time $\lb 0,L\rb$ into  subintervals $I_j=\lb j \Delta^{2+\eta}, (j+1) \Delta^{2+\eta}\rb$ of length $\Delta^{2+\eta}$, for $j \in \lb 0, \frac{L}{\Delta^{2+\eta}} -1 \rb$.  We say that $I_j$ is bad if it contains a time $\ell$ such that $\inf_{m\in \lb \ell,L \rb} Y_m < Y_\ell-\Delta$. Otherwise, $I_j$ is called good. We call $I_{\text{good}}$ the union of good intervals $I_j$ and let $M_j$ be the smallest multiple of $\eta \Delta$ greater than $\max_{\ell\in \lb 0, (j+1) \Delta^{2+\eta}\rb \cap I_{\text{good}}} Y_\ell$,
  with the convention that $\sup_\emptyset:=0$.  We let $J$ be the set of $j\in \lb0,\frac{L}{\Delta^{2+\eta}}-1\rb$ such that $I_{j+1}$ is bad {\it or} $M_{j+1}\neq M_{j}$.

  If $j\notin J$, $M_{j+1}=M_j$ and $I_{j+1}$ is good. By definition of $M_j$, there is some $\ell \le (j+1) \Delta^{2+\eta}$ such that $M_j < Y_\ell +\eta \Delta$ and $\inf_{m \in \lb\ell,L\rb} Y_m \ge Y_\ell -\Delta$. We deduce that $\inf_{m\in I_{j+1}} Y_m \ge M_j - (1+\eta)\Delta=M_{j+1}-(1+\eta)\Delta$.

On the event $\{ \max_{j\in \lb 0, L\rb} Y_j \le \frac{L}{\Delta^{1+c}}     ,\, m_\Delta(L) \le \frac{L}{\Delta^{2+c}} \}$ as in the probability term of Lemma \ref{main2:upper}, there are at least $\frac{L}{\Delta^{2+\eta}}-\frac{L}{\eta \Delta^{2+c}}$ intervals for which $M_j=M_{j+1}$. Moreover, there are less than $\frac{L}{\Delta^{2+c}}$ bad intervals. Hence $\#J^c\ge  \frac{L}{\Delta^{2+\eta}}-\frac{L}{\eta \Delta^{2+c}}- \frac{L}{\Delta^{2+c}} \ge \frac{L}{(1+\eta)\Delta^{2+\eta}} $ if  $\Delta\ge \Delta_0$ for some large enough $\Delta_0$.  Exactly as in the proof of Lemma \ref{main:upper},  we apply the Markov property at time $ (j+1)\Delta^{2+\eta}$ for each $j \not\in J$ (in decreasing order of $j$), ignore those  $j\in J$, then apply \eqref{mogulskii_upper0} to $\upsilon=\frac\eta2$ and $ d (1+\eta)^4< \frac{\pi^2}{2}$ to obtain that $$ \p\Big( \max_{j\in \lb 0, L\rb} Y_j \le \frac{L}{\Delta^{1+c}}     ,\, m_\Delta(L) \le \frac{L}{\Delta^{2+c}}, J=J_0\Big)
 \le
\Big( e^{- d (1+\eta)^2 \Delta^\eta}\Big)^{\#J_0^c}
\le e^{- d (1+\eta) \frac{L}{\Delta^2}} ,
$$

\noindent by using the fact that $\#J_0^c \ge \frac{L}{(1+\eta)\Delta^{2+\eta}}$.  Finally, there are fewer than $2^{\frac{L}{\Delta^{2+\eta}}}$ possible choices for the set $J$. The lemma then follows by a  union bound.
  \end{proof}

\medskip

We end this section by an elementary fact about the order statistics of a family of i.i.d. exponential random variables.

\begin{lemma}\label{l:upper_stoc}
Let $(E_i)_{1\le i\le \ell}$ be $\ell$ i.i.d. exponential r.v. with mean $1$. Let $E_{(1)}>E_{(2)}>\ldots >E_{(\ell)}$ be their decreasing order statistics. Then for any $k\in \{1,\ldots, \ell\}$, 
\[
\e\Big[\sum_{i=1}^{k} E_{(i)}\Big]\le k +  k \ln \frac{\ell}{k}.
\]
\end{lemma}
\begin{proof}
    We observe that for any $x>0$,
    \[
    \sum_{i=1}^k {\bf 1}_{\{E_{(i)} \ge x \}} = \min(\sum_{i=1}^\ell {\bf 1}_{\{E_i \ge x\}},k). 
    \]
    We bound the right-hand side by $\sum_{i=1}^\ell {\bf 1}_{\{E_i \ge x\}}$ if $\ell e^{-x}< k$ and by $k$ otherwise. Taking the expectation and integrating over $x$ yields the result.
\end{proof}

\subsection{Results on branching random walks: Proof of Proposition \ref{p:first_moment}} \label{sub:brw}
Recall that $\V$ is a one-dimensional branching random walk  whose  distribution of the displacements  is given by \eqref{eq:BRW-V}. We suppose that it starts with one particle. The many-to-one lemma \cite[Section 1.3]{zhan}  reads as follows:
\begin{equation}\label{eq:many-to-one_V}
\e[\sum_{|x|=\ell} e^{-\V(x)}F(\V(x_1),\V(x_2),\ldots,\V(x_\ell))]= \e[F(Y_1,\ldots,Y_\ell)],
\end{equation}

\noindent where as before, $|x|$ denotes the generation of $x$ in the genealogical tree of $\V$, and for any $1\le i \le \ell$, $x_i$ denotes the ancestor  of $x$ at generation $i$ (so $x_\ell=x$ when $|x|=\ell$). 

We will use the following lemma for second moment computations. 

\begin{lemma}
Let  $\upsilon>0$, $0<d<\frac{\pi^2}{2}$, and let $\Delta$,  $r_j\le s_j$  be as in Lemma \ref{mogulskii_upper}. For any $\ell\ge 1$ and $\delta>r_\ell$,  we have \begin{equation}\label{second_moment}
\e\Big[\Big(\sum_{|x|=\ell} {\bf 1}_{\{r_j  \le \V(x_j) \le s_j,\, \forall\, 1\le j\le \ell\}}{\bf 1}_{\{\V(x_\ell)\le \delta\}}\Big)^2\Big] \le 
 e^{3\Delta^\upsilon+2 \delta-2d \frac{\ell-1}{\Delta^2}}\sum_{m=0}^{\ell-1} e^{ d \frac{ m}{\Delta^2}- r_m} + e^{ \delta +\Delta^\upsilon-\frac{d}{\Delta^2}\ell}. 
\end{equation}
\end{lemma}

\begin{proof} For any $|x|=\ell\ge 1$, we let
\begin{equation*} 
    F_x={\bf 1}_{\{r_j  \le \V(x_j) \le s_j,\, \forall\, 1\le j\le \ell\}}{\bf 1}_{\{\V(x_\ell)\le \delta\}}.
\end{equation*}

\noindent Discussing on the ancestor of $x$ and $y$, 
\[
\Big(\sum_{|x|=\ell} F_x\Big)^2=\sum_{|x|=|y|=\ell} F_xF_y
 =
\sum_{m=0}^{\ell-1} 
\sum_{|z|=m} 
\sum_{z_i\neq z_j} \Sigma(z_i)\Sigma(z_j) 
 + \sum_{|x|=\ell} F_x^2 
\]
\noindent where $z_i$ is the $i$-th child of $z$ and for $|u|=m+1\le \ell$,
\[
\Sigma(u) = \sum_{\substack{|x|=\ell, x_{m+1}=u}}  F_x.
\]
\noindent  Then 
\[
\Sigma(u)= {\bf 1}_{\{ \V(u_j) \in [r_j,s_j],\, \forall\, 1\le j\le m\}} \widetilde{\Sigma}(u) 
\]
with 
\[
\widetilde{\Sigma}(u) := \sum_{\substack{|x|=\ell, x_{m+1}=u}}  {\bf 1}_{\{ \V(x_j) \in [r_j,s_j],\, \forall\, m < j \le \ell,\, \V(x) \le \delta \}}.
\]

\noindent  Let  ${\mathcal F}_i:=\sigma\{\V(v),\, |v|\le i\}$ for $i\ge 0$. By the branching property and \eqref{eq:many-to-one_V},  \begin{align*}    
\e[\widetilde{\Sigma}(u) \mid {\mathcal F}_{m+1}, \V(u)=a] 
&= \e \Big[\sum_{|x|=\ell-m-1} {\bf 1}_{\{\V(x_i) +a \in [r_{i+m+1}, s_{i+m+1}], \forall \, i\le \ell-m-1, \V(x) +a \le \delta\}}\Big]
\\
&= \e\Big[e^{Y_{\ell-m-1}} {\bf 1}_{\{Y_i +a \in [r_{i+m+1}, s_{i+m+1}], \forall \, i\le \ell-m-1, Y_{\ell-m-1} +a \le \delta\}}\Big]
\\
&\le e^{\delta-a} \,e^{\Delta^\upsilon - d \frac{\ell- m-1}{\Delta^2}}, 
\end{align*}

\noindent where in the above inequality, we have used $Y_{\ell-m-1}\le \delta-a$ then applied Lemma \ref{l:upper}.  We obtain
\begin{align*}
\e\Big[\sum_{z_i\neq z_j} \widetilde{\Sigma}(z_i)\widetilde{\Sigma}(z_j) \mid {\mathcal F}_m \Big] &\le  e^{2\Delta^\upsilon - 2 d \frac{\ell- m-1}{\Delta^2}}e^{2\delta}\e\Big[\sum_{z_i\neq z_j} e^{- (\V(z_i)+\V(z_j))} \mid \V(z)\Big] \\
&= 
 e^{2\Delta^\upsilon - 2 d \frac{\ell- m-1}{\Delta^2}+2(\delta-\V(z))},  
\end{align*}

\noindent where in the last inequality, we have used the fact that $\e [\sum_{|u|=|u'|=1, u\neq u'} e^{- (\V(u)+\V(u')}]= 1$.   We end up with
\begin{align*}
\e[\sum_{m=0}^{\ell-1} 
\sum_{|z|=m} 
\sum_{z_i\neq z_j} \Sigma(z_i)\Sigma(z_j)] &\le e^{2\Delta^\upsilon+2 \delta} \sum_{m=0}^{\ell-1} e^{- 2 d \frac{\ell- m-1}{\Delta^2}}
\e[ \sum_{|z|=m}  e^{-2 \V(z)} {\bf 1}_{\{ \V(z_j) \in [r_j,s_j],\, \forall\,   j\in \lb 1, m\rb\}}]\\
&\le  e^{2\Delta^\upsilon+2  \delta}\sum_{m=0}^{\ell-1} e^{- 2 d \frac{\ell- m-1}{\Delta^2} -   r_m }\e[ \sum_{|z|=m}  e^{- \V(z)} {\bf 1}_{\{ \V(z_j) \in [r_j,s_j],\, \forall\,  j\in \lb 1, m\rb\}}].
\end{align*}

\noindent By \eqref{eq:many-to-one_V}, \[
\e[ \sum_{|z|=m} e^{- \V(z)} {\bf 1}_{\{ \V(z_j) \in [r_j,s_j],\, \forall\,  j\in \lb 1, m\rb\}}] =  \p(Y_j\in [r_j,s_j],\, \forall\,  j\in \lb 1, m\rb)\le e^{\Delta^\upsilon-d\frac{m}{\Delta^2}}.
\]

\noindent This gives the first term in the right-hand-side of \eqref{second_moment}. Similarly 
\[
\e[\sum_{|x|=\ell} F_x^2] = \e[\sum_{|x|=\ell} F_x] \le e^{ \delta +\Delta^\upsilon-\frac{d}{\Delta^2}\ell}, 
\]

\noindent which gives the second term in the right-hand-side of \eqref{second_moment} and  completes the proof.
\end{proof}

\begin{lemma}\label{l:barrier}  Let $c>0$. For $L\ge 1$, 
 let $r_j=  \frac{\pi^2}{2 c^2} \, L^{-2/3}j$ and $s_j = r_j + c \,  L^{1/3} $ for $j\ge 0$.  For any   $\upsilon>0 $, we have   \begin{equation}\label{eq:barrier}
\lim_{L\to\infty} L^{-1/3} \ln \p\Big(\exists |x|=L: \V(x_j)\in [r_j,s_j],\,\forall\, j\le L,\, \V(x)\le r_{L} + \upsilon L^{1/3}  \Big)= 0.
\end{equation}
\noindent Moreover,
\begin{equation}\label{bound_blue}
\lim_{L\to\infty} L^{-1/3} \ln \e\Big[\sum_{|x|\le L} {\bf 1}_{\{\V(x_j)\in [r_j,s_j],\,\forall\, j\le |x|\}}\Big] =c.
\end{equation}
\end{lemma}

\begin{proof}  Let for $|x|=L$, 
\[
F_x:={\bf 1}_{\{ \V(x_j)\in [r_j,s_j]\,\forall\, j\le L,\, \V(x)\le r_{L} + \upsilon L^{1/3}  \}}.
\]

\noindent By \eqref{eq:many-to-one_V}, then applying \eqref{mogulskii_lower} with $\Delta=   c L^{1/3}$, we obtain that for any $d'>\frac{\pi^2}{2}$ and all large $L$,
\begin{align*}    
\e\Big[\sum_{|x|= L} F_x\Big]
& \ge e^{r_L} \p(Y_j\in [r_j,s_j],\, \forall\,j\le L, Y_L\le r_{L} +\upsilon L^{1/3})
\\
& \ge e^{r_L} e^{- \frac{d'}{ c^2} L^{1/3}} = e^{- \varepsilon' L^{1/3}},
\end{align*}

\noindent with $\varepsilon':=(d'-\frac{\pi^2}{2}) c^{-2}$. Similarly, by applying \eqref{mogulskii_upper}, we obtain that for any $d<\frac{\pi^2}{2}$,
\[
\e\Big[\sum_{|x|=L} F_x\Big] \le  e^{(  \upsilon+ \varepsilon) L^{1/3}}
\]
where $\varepsilon:=2 (\frac{\pi^2}{2}-d)c^{-2}$, and the factor $2$ in $\varepsilon$ is included to cancel the term $\Delta^\upsilon$ in \eqref{mogulskii_upper}. %Both $\varepsilon$ and $\varepsilon'$ can be made arbitrarily small.

We compute now the second moment of $\sum_{|x|=L} F_x$ using \eqref{second_moment}. With  $\Delta= c L^{1/3}$, we apply \eqref{second_moment} with $\delta= r_L+ \upsilon L^{1/3}$, and $\ell=L$. We check that for $m\le L$,
\[
\frac{d}{\Delta^2}m - r_m = -\frac{\varepsilon}{2}  \frac{m}{L^{2/3}}, \qquad \delta -\frac{d}{\Delta^2}L= (\frac\varepsilon2+\upsilon) L^{1/3}.
\]

\noindent Then for all large $L$, 
\[
\e\Big[(\sum_{|x|=L} F_x)^2\Big]\le e^{ (2 \upsilon+ 2 \varepsilon) L^{1/3}}.
\]

\noindent We deduce from the Paley-Zygmund inequality that
\[
\p\Big(\sum_{|x|= L} F_x \ge 1\Big) \ge \frac{\e[\sum_{|x|=L} F_x]^2}{\e[(\sum_{|x|=L} F_x)^2]}\ge e^{ -2 ( \upsilon+  \varepsilon+ \varepsilon') L^{1/3}}.
\]

\noindent We recall that $\varepsilon$ and $\varepsilon'$ can be arbitrary small. Moreover, the probability on the left-hand side in \eqref{eq:barrier} is decreasing when $\upsilon\to 0$. We deduce \eqref{eq:barrier}.

 Let us prove \eqref{bound_blue}. Let  $\ell\le L$. By \eqref{eq:many-to-one_V} and \eqref{mogulskii_upper}, we obtain that   for any $d\in (0, \frac{\pi^2}2)$ and $\upsilon\in (0, 1)$, 
\begin{align*}     
\e\Big[\sum_{|x|=\ell} {\bf 1}_{\{\V(x_j)\in [r_j,s_j], \, \forall\, j\le \ell\}}\Big]
&\le e^{ s_\ell}\p(Y_j\in [r_j,s_j], \, \forall\, j\le \ell)
\\
& \le e^{ s_\ell} e^{ -\frac{d}{c^2 L^{2/3}} \ell  + L^{\upsilon/3}}
= e^{\frac\varepsilon2 \frac{\ell}{L^{2/3}} + L^{\upsilon/3}+ c L^{1/3}}.
\end{align*}

\noindent Take the sum over $\ell$ gives that $$ \e\Big[\sum_{|x|\le L} {\bf 1}_{\{\V(x_j)\in [r_j,s_j], \, \forall\, j\le |u|\}}\Big]
\le
L\,  e^{\frac\varepsilon2 L^{1/3} + L^{\upsilon/3}+ c L^{1/3}} .$$

\noindent Since $\varepsilon$ can be taken arbitrarily close to $0$, we deduce the upper bound of \eqref{bound_blue}.

The lower bound of \eqref{bound_blue} is proved similarly, using 
\begin{align*}
& \e\Big[\sum_{|x|=L} {\bf 1}_{\{\V(x_j)\in [r_j,s_j],\, \forall\, j\le L,\, \V(x) \in [s_L-\upsilon  L^{1/3}, s_L]\}}\Big]\\ 
&\ge e^{(s_L-\upsilon L^{1/3})}\p(Y_j\in [r_j,s_j],\, \forall\, j\le L,\,Y_\ell \in [s_L-\upsilon  L^{1/3} ,s_L] )
\end{align*}
and then Lemma \ref{l:mogulskii_lower} with $\Delta= c L^{1/3}$,
\[
\p(Y_j\in [r_j,s_j],\, \forall\, j\le L,\, Y_L \in [s_L-\upsilon  L^{1/3},s_L])
\ge e^{-\frac{d'}{c^2} L^{1/3} }.
\]
Letting $d'\to \frac{\pi^2}{2}$ and $\upsilon\to 0$ gives the lower bound.
\end{proof}

\begin{lemma}\label{l:beginning}
Let $b>\upsilon>0$, $r:=-b\ln n$.  We have  
\begin{equation}   \label{atleast1-begining} 
\lim_{n\to\infty} \frac{1}{\ln n} \ln \p\Big(\exists |x|= \lfloor \ln^{5/2} n\rfloor:  \V(x_j) \in [r,0],\, \forall\, j\le |x|,\, \V(x) \le r+\upsilon \ln n\Big)=\upsilon-b.
\end{equation}
Moreover, 
\begin{equation}\label{bound_beginning}
\e\Big[\sum_{\ell=1}^\infty \sum_{|x|=\ell} {\bf 1}_{\{\V(x_j)\in [r,0],\, \forall\, j\le \ell\}}\Big] = O(\ln n)
\end{equation}
and \begin{equation}\label{bound_beginning2}
   \e\Big[\sum_{|x|=\lfloor \ln^{5/2}n\rfloor} {\bf 1}_{\{\V(x_j)\in [r,0],\, \forall\, j\le |x|,\, \V(x)\le r+\upsilon \ln n\}}\Big]= n^{-b+\upsilon+o(1)}. 
\end{equation}
\end{lemma}
\begin{proof}
Equation \eqref{bound_beginning2} is a consequence of \eqref{eq:many-to-one_V}. The upper bound comes from \eqref{mogulskii_upper}, noting that $e^{Y_{\lfloor \ln^{5/2} n\rfloor}} \le n^{-b+\upsilon}$. For the lower bound, we  restrict to $\V(x)\ge  r':=r+\upsilon' \ln n$ for some $\upsilon'<\upsilon$ close to $\upsilon$ so that $e^{Y_{\lfloor \ln^{5/2} n\rfloor}} \ge n^{-b+\upsilon'}$ and use  \eqref{mogulskii_lower}.  Equation  \eqref{bound_beginning2} gives the upper bound of \eqref{atleast1-begining}. For the lower bound, we restrict to  $\V(x_j) \in [r',0]$. We then use \eqref{second_moment} as in the proof of the previous lemma and the Paley--Zygmund inequality. We omit the details.
 Finally, for \eqref{bound_beginning}, we write $\e[\sum_{|x|=\ell} {\bf 1}_{\{\V(x_j)\in [r,0],\, \forall\, j\le \ell\}} ]\le \p(Y_j \in [r,0],\, \forall\, j\le \ell)$ whose sum over $\ell$ is bounded by  the expectation of the first exit time of $[r, 0]$ by the centered random walk $Y$ starting from $0$, and  is therefore  $O(|r|)$.  %The proof of \eqref{bound_beginning2} is similar, noting that $e^{Y_{\lfloor \ln^{5/2} n\rfloor}} \le n^{-b+\upsilon+o(1)}$.
\end{proof}

 Recall that $\V^N$ denotes the branching random walk $\V$ in which we select the $N$ leftmost particles of the population at each step. The following coupling is a variant of \cite[Lemma 1]{berard_gouere}.
We say that a branching random walk is killed   if at each generation, the particles produce offspring as in $\V$, which are then selected according to a rule which is adapted to the natural filtration generated by $\V$.  

\begin{lemma}\label{l:kil}   Let $\V^{\rm kil}$ be the branching random walk $\V$ killed according to a given rule. Assume that    at time $0$, the population of $\V^{\rm kil}$ is included in (or equal to) that of $\V^N$. Then there exists a coupling of   $\V^N$ and $\V^{\rm kil}$ such that $$\min_{|u|=k} \V^N(u) \le \min_{|u|=k} \V^{\rm kil}(u), \qquad k\le T,$$ where  $T:= \inf\{k\ge 0: \#\V_k^{\textrm{kil}}> N\}.$  \end{lemma}

We use the convention $\min_\emptyset:=\infty$ and note that the way in which $\V^{\rm kil}$ is killed is irrelevant for the above lemma.

\begin{proof}   For two finite counting measures $\mu= \sum_{i=1}^{a(\mu)} \delta_{\{x_i\}}$ and $\nu= \sum_{j=1}^{a(\nu)}\delta_{\{y_j\}}$ on $\r$, we assume that $x_1\le x_2 \le  \cdots\le  x_{a(\mu)}$ and $y_1\le y_2\le \cdots \le y_{a(\nu)}$.  We write $\mu \succ \nu$ if $a(\mu)\le a(\nu)$ and $x_i \ge y_i $ for all $i \in \lb 1, a(\mu)\rb$. In the degenerate case $a(\mu)=0$, we always have $\mu \succ \nu$ for any $\nu$. For any $k\ge 0$, we identify  $\V_k^{\rm kil}$ and $\V_k^N$ 
with their associated finite counting measures. It suffices  to construct a coupling  of   $\V^N$ and $\V^{\rm kil}$ such that \begin{equation}    \V_k^{\rm kil} \succ \V^N_k, \qquad \forall\, 0\le k< T. \label{compar-kil-N}\end{equation}

By assumption, \eqref{compar-kil-N} holds for $k=0$.  Assume the coupling has been built up to some time $k< T$. We may write $\V_k^{\rm kil}= \sum_{i=1}^{a_k} \delta_{\{x_i\}}$ and $\V^N_k= \sum_{j=1}^{b_k}\delta_{\{y_j\}}$ such that $a_k \le b_k$ and $x_i \le y_i$ for all $i\in \lb 1, a_k\rb$.  If $a_k=0$,   $T=\infty$ and \eqref{compar-kil-N} holds for all time after $k$ by definition. Consider the non-degenerate case $a_k \ge 1$. For $1\le i \le a_k$, each $x_i$ particle  of $\V^{\rm kil}_k$  produces  a collection of offspring $x_i+ \Xi_{i, k}$ at time $k+1$ in $\V$, where $(\Xi_{i, k})_{i, k\ge 1}$ are i.i.d. copies of $
\sum_{i=1}^\infty {\delta}_{\{e T_i -1\}} $ and for any $z\in \r$, $z+ \Xi_{i, k}$ is the counting measure $\Xi_{i, k}$ shifted by $z$.  Those offspring that are non-killed (if the set is not empty)  constitute $\V_{k+1}^{\rm kil}$. For $\V^N$, each $y_i$, $1\le i\le  b_k$,  produces  a collection of offspring $y_i+ \Xi_{i, k}$ at time $k+1$, and we select the $N$-leftmost particles to form $\V^N_{k+1}$ (in particular $\#\V^N_{k+1}=N$).  If the population of $\V_{k+1}^{\rm kil}$ does not exceed $N$, then necessarily $\V_{k+1}^{\rm kil} \succ \V^N_{k+1}$, and \eqref{compar-kil-N} holds at the time $k+1$. Otherwise, $T=k+1$, and  \eqref{compar-kil-N} holds as well. 
\end{proof}

Now we are ready to prove Proposition \ref{p:first_moment}.

\begin{proof}[Proof of Proposition \ref{p:first_moment}]  Fix $a\in [0, 1]$ and $t>0$. Let $\ell_N$ be such that $\frac{\ell_N}{\ln^3 N} \to t$.  Proposition \ref{p:first_moment} follows if we can show  that   \begin{align}    
   & \lim_{N\to\infty} \frac{1}{\ln N} \e[ \min_{|u|=\ell_N  } \V^N(u) ]  = -a+\frac{\pi^2}{2}t, \label{NBrw-exp}
    \\
    & \lim_{N\to\infty} \frac{1}{(\ln N)^2} \mbox{\rm Var}\Big(\min_{|u|=\ell_N} \V^N(u)\Big)   =0. \label{NBrw-var}
\end{align}  

It is enough to show that for any $\zeta < -a+\frac{\pi^2}{2}t< \zeta'$, there exists some $\varepsilon>0$ such that for all large $N$, \begin{align}   \p\Big( \min_{|u|=\ell_N  } \V^N(u) \ge  \zeta' \ln N\Big) & \le N^{-\varepsilon}, \label{VN-up}
\\
    \p\Big( \min_{|u|=\ell_N  } \V^N(u)   \le  \zeta  \ln N\Big) & \le N^{-\varepsilon}. \label{VN-low}\end{align}

Indeed,  observe that $\min_{|u|=\ell_N} \V^N(u)$ is stochastically smaller than $\V^1$ at time $\ell_N$, where $\V^1$ is the branching random walk $\V$  with selection, in which only the leftmost particle is kept at each step. Therefore $\V^1$ is a random walk with step distribution $e \mbox{Exp}(1)-1$, and thus $\e[(V^1_{\ell_N})^4] = O(\ell_N^4)$. Since $\min_{|u|=\ell_N} \V^N(u)> - \ell_N$, we get that \begin{equation}   \e[(\min_{|u|=\ell_N} \V^N(u))^4] \le  \ell_N^4 + \e[(V^1_{\ell_N})^4] = O(\ell_N^4). \label{moment4}  \end{equation}

 Let $\H_N:= \{\zeta \ln N \le \min_{|u|=\ell_N  } \V^N(u) \le \zeta' \ln N\}$.  By \eqref{VN-up} and \eqref{VN-low}, the Cauchy--Schwarz inequality implies that 
\begin{equation*}     
 \lim_{N\to\infty}  \e[ (\min_{|u|=\ell_N} \V^N(u))^2 {\bf 1}_{\H_N^c}] = 0,
\end{equation*}

\noindent which, in view of the fact that on $\H_N$, $\zeta \ln N \le \min_{|u|=\ell_N  } \V^N(u) \le \zeta' \ln N$, yields \eqref{NBrw-exp} and \eqref{NBrw-var}.  

 In what follows, we prove \eqref{VN-up} and \eqref{VN-low} in two separate parts. 
\medskip

 {\it (i) Proof of \eqref{VN-up}.}     We treat at first the case $a\in (0, 1)$.  By Lemma \ref{l:kil}, we will construct a killed branching random walk $\V^{\rm kil}$ which, w.h.p., has its minimum  below $\zeta' \ln N$ at time $\ell_N$ and $T> \ell_N$, where $T$ denotes the first time that the size of  $\V^{\rm kil}$  exceeds $N$.

 Choose $c\in (0, t^{-1/3})$,   $b\in (0,a)$,  and $\upsilon>0$ such that $$c \,t^{1/3}+a-b+\upsilon<1, \qquad -b + \frac{\pi^2}{2c^2}t^{1/3} + 2 \upsilon (1+t^{1/3}) < \zeta',$$
 
 \noindent where the second inequality is possible because  $-b + \frac{\pi^2}{2c^2}t^{1/3} + 2 \upsilon (1+t^{1/3}) \to - a+ \frac{\pi^2}{2} t< \zeta' $ as $b \uparrow a, \upsilon\downarrow 0$ and $c \uparrow t^{-1/3}$.

   From time $0$ to $j_0:=\lfloor \ln^{5/2} N\rfloor$, one kills particles outside $[-b\ln N,0]$. Let $L:= \ell_N - j_0$.  We consider alive particles at time $j_0$ at position smaller than $-b\ln N + \upsilon \ln N$. For each such particle, say at position $y\in [-b\ln N,(-b+\upsilon)\ln N]$, one looks at their descendants at generation $\ell_N$ killed outside the interval $y+[r_j,r_j + c L^{1/3}]$  for any  time $j_0\le j\le \ell_N$ where $r_j:=\frac{\pi^2}{2c^2} L^{-2/3}(j-j_0)$. 
It defines the killed branching random walk $\V^{\rm kil}$. By Lemma \ref{l:kil}, we have
\begin{equation}\label{eq:coupling}
  \min_{|u|=\ell_N} \V^N(u)   \le   \min_{|u|=\ell_N} \V^{\rm kil}(u), \qquad \mbox{on $\{T> \ell_N\}$}.   
  \end{equation}
 
 We start with $N^a$ particles at time $0$ for {\it both $\V^N$ and $\V^{\rm kil}$.} 
 By \eqref{bound_beginning}, the expected number of particles alive till time $j_0$ is smaller than $O(N^a \ln N)= N^{a+o(1)}$, where the factor $N^a$ corresponds to the number of particles at time $0$.     Applying \eqref{bound_beginning2}   and \eqref{bound_blue}, the expected number of particles alive between time $j_0$ and $\ell_N$ is less than $N^a \times N^{-b+\upsilon+o(1)}\times e^{ (c+o(1)) L^{1/3}}= N^{a-b+\upsilon + c t^{1/3}+o(1)}$. Let $\eta:= \max(a, a-b+\upsilon+c \,t^{1/3})$. Recall that $\eta<1$. We have proved that the expected number of particles alive up to $\ell_N$ is less than $N^{\eta + o(1)}$. Then we deduce from the Markov inequality that \begin{equation}    \p(T\le \ell_N)\le N^{-(1-\eta)+o(1)}. \label{probaT<ellN} \end{equation}

 Let $\delta \in (0,\frac{a-b+\upsilon}2)$.  Recall that we start with $N^a$ particles at time $0$. By \eqref{atleast1-begining} (see Lemma \ref{l:beginning}), the number of particles staying in $[-b\ln N,0]$ ending below $-b\ln N+\upsilon \ln N$ at time $ j_0$ is stochastically larger than a binomial random variable with parameters $N^a$ and the success probability $N^{-(b-\upsilon)+o(1)}$. Then,  for some $\varepsilon>0$, with probability at least $1-e^{-N^{\varepsilon+o(1)}}$, there are at least $N^\delta$ particles staying in $[-b\ln N,0]$ ending below $-b\ln N+\upsilon \ln N$ at time $ j_0$.  

By the branching property and \eqref{eq:barrier}, conditioning on there being at least $N^\delta$ particles in $[-b \ln N, - (b-\upsilon) \ln N]$ at time $j_0$,   we deduce that with probability greater than $1-(1-e^{-o(L^{1/3})})^{N^\delta}\ge 1-e^{-N^{\delta+o(1)}}$, there will be an alive particle $u$ at time $\ell_N$ with position smaller than $y+\frac{\pi^2}{2c^2} L^{1/3} + \upsilon L^{1/3}$ where $y$ is the position of the ancestor of $u$ at time $j_0$, hence $y\le -b\ln N+\upsilon\ln N$. Therefore for all large $N$, $$\V^{\rm kil}(u)\le -(b-\upsilon)\ln N +\frac{\pi^2}{2c^2} (\ell_N)^{1/3}+ \upsilon (\ell_N)^{1/3}\le [-b + \frac{\pi^2}{2c^2}t^{1/3} + 2 \upsilon (1+t^{1/3})]  \ln N< \zeta' \ln N.$$

 Let $\varepsilon':=\min(\delta,\varepsilon)$. Then we have proved that there exists an event $E_N$ with $\p(E_N)\ge 1- e^{-N^{\varepsilon'+o(1)}}$ such that
$$
  \min_{|u|=\ell_N} \V^{\rm kil}(u) < \zeta' \ln N, \qquad \mbox{on  $E_N$}.$$
 
 This and \eqref{probaT<ellN} yield \eqref{VN-up} for the case $a \in (0, 1).$

The case $a=1$ in \eqref{VN-up} follows from the case $a<1$ since starting with fewer particles stochastically increases the minimum of $\V^N$.  

For the case $a=0$ in \eqref{VN-up}, we consider the binary branching random walk starting with one particle with offspring law given by  $(eT_1 -1, e T_2-1)$, i.e. the two leftmost children in $\V$.   Let $\Upsilon_k$ be the the maximum of the binary branching random walk at time $k\ge 1$. By the union bound,  for any $c >0$, $\p (\Upsilon_k \ge c k   ) $ is at most $2^k$ times the probability that the sum of $k$ i.i.d. copies of $e T_2- 1$ exceeds  $c k$.  Since $e T_2- 1$ has some exponential moments,   standard large deviation estimates imply that we can choose and then fix $c$ sufficiently large   so that this latter probability is at most $2^{-k} e^{- 2 k}$ for all large $k$. Hence 
  $$\p (\Upsilon_k \ge c k ) \le e^{- 2 k}.$$

Let $\zeta'> \frac{\pi^2}{2}t$ and  choose $\delta,\delta'\in (0,1)$ such that $\delta < (\zeta' - \frac{\pi^2}{2}t)/c$ and $e^{\delta'}<2^{\delta}$. Then for all large $N$, \begin{equation}    \p\Big(\Upsilon_{\lfloor \delta \ln N\rfloor} \ge c \delta \ln N\Big) \le N^{-\delta}. \label{Upsilon-max} \end{equation}

 At time $\lfloor \delta \ln N\rfloor$, there are more than $N^{\delta'}$ particles in the population of the above binary branching random walk.   Then $\min_{|u|=\ell_N} \V^N(u)$ is stochastically smaller than $\Upsilon_{\lfloor \delta \ln N\rfloor}+ {\mathcal M}_N$, where $ {\mathcal M}_N$ denotes  the minimum of $\V^N$ at time $\ell_N-\lfloor \delta \ln N\rfloor$ starting with $N^{\delta'}$ particles.  Applying the already proved \eqref{VN-up} with $a=\delta'$,  we obtain some $\varepsilon>0$ such that  $$\p\Big( {\mathcal M}_N > \frac{\pi^2}{2}t \ln N \Big) \le N^{-\varepsilon}.$$
 
 \noindent By the triangular inequality, this bound together with \eqref{Upsilon-max} imply \eqref{VN-up} in the case $a=0$. 
  
\medskip
{\it (ii) Proof of \eqref{VN-low}.}  Let   $b >1$ in this part.  First,  we show that w.h.p., no drop  larger than $ b \ln N$ occurs in $\V^N$ up to time $\ell_N$.    More precisely, let $F_N$ be the event that there exists some particle $v$ in $\V^N$ and an ancestor $u$ of $v$  such that $\V^N(u) - \V^N (v) > b \ln N$. Here $\V^N$ is obtained by selecting  particles in the $\lfloor N^a\rfloor$ independent copies of branching random walks starting at $0$; the genealogy of the selected particle $v$ is understood with respect to its underlying branching random walk.

By taking into account  all possible choices of $u$, at most $N$ at each time, we  deduce from the union bound and the branching property at $u$ that $$\p(F_N) \le  N\, \sum_{j=0}^{\ell_N-1} \p\Big(\exists v: |v|\le \ell_N-j, \V(v_i)\ge - b \ln N, \forall\, i < |v|, \V(v) < - b \ln N\Big),$$

\noindent where in the probability term on the right-hand side of the above inequality, $\V$ starts with $1$ particle located at $0$. For any $j \le \ell_N-1$, by summing over $|v|$ then applying \eqref{eq:many-to-one_V},   each probability term is  bounded by \begin{align*}  &    \sum_{k=0}^{\ell_N-j} \e \Big[\sum_{|v|=k} {\bf 1}_{\{\V(v_i)\ge - b \ln N, \forall\, i < k, \,\V(v) < - b \ln N\}}\Big]
\\
& \le N^{-b} \sum_{k=0}^{\ell_N-j}\p\Big(Y_i \ge - b \ln N, \forall\, i < k, \,  Y_k< - b \ln N\Big)
\le N^{-b}.
 \end{align*}

 Hence we have  \begin{equation} \label{eq:drop}    \p(F_N) \le N^{1-b} \ell_N = N^{-(b-1)+o(1)}.  \end{equation}

 Let $\zeta < -a+\frac{\pi^2}{2}t$.  We have \begin{align*}   & \p\Big(\min_{|u|=\ell_N} \V^N(u) < \zeta \ln N, F_n^c\Big)
 \\
 &\le N^a \, \e\Big[ \sum_{|u|=\ell_N} {\bf 1}_{\{ \V(u) < \zeta\ln N,\,  \min_{0\le i \le j\le \ell_N} (\V(u_j) - \V(u_i)) \ge - b \ln N\}}\Big]
 \\
 &\le  N^a\, N^{\zeta}\, \p\Big( Y_{\ell_N} < \zeta\ln N, \min_{0\le i \le j\le \ell_N} (Y_j - Y_i) \ge - b \ln N\Big),
 \end{align*}
 
 \noindent where in the first inequality $\V$ denotes again the branching random walk started with $1$ particle located at $0$, and we have applied  \eqref{eq:many-to-one_V} to obtain the second inequality. By Lemma \ref{main:upper}, for any $d \in (0, \frac{\pi^2}{2})$ and large $N$, $$ \p\Big( Y_{\ell_N} < \zeta \ln N, \min_{0\le i \le j\le \ell_N} (Y_j - Y_i) \ge - b \ln N\Big) \le e^{- d \frac{\ell_N}{(b\ln N)^2}}= N^{- \frac{d t}{b^2} + o(1)}. $$
 
 \noindent Hence $$ \p\Big(\min_{|u|=\ell_N} \V^N(u) < \zeta \ln N, F_n^c\Big) \le N^{ a+\zeta - \frac{d t}{b^2} + o(1)}.$$
 
 For a given $\zeta < -a+\frac{\pi^2}{2}t$,   choose $b>1$ sufficiently close to $1$ and $d$ sufficiently close to $\frac{\pi^2}{2}$ so that $a+\zeta - \frac{d t}{b^2} <0$. This, together with \eqref{eq:drop},  yields   \eqref{VN-low} and completes the proof of Proposition \ref{p:first_moment}.
       \end{proof}

We conclude this section with a consequence of Proposition \ref{p:first_moment}, which was used in the proof of Lemma \ref{lem:low-T2}.

\begin{corollary}  \label{coro:concentration}  Let  $(\Theta_i)_{i\ge 1}$ be a sequence of independent copies of $ \min_{|u|=\lfloor \ln^3 N\rfloor-1} \V^N(u)$, where $\V^N$ starts with $1$ particle located at $0$. Fix $c>1$. For any $\delta>0$,   $$\sup_{m\ge 1} \p\Big( \cup_{k=m}^{\lceil cm \rceil} \{| \sum_{i=1}^k (\Theta_i- \e(\Theta_i)) | \ge \delta k \, \e(\Theta_1)\}\Big) \to 0, \qquad N\to\infty.$$
\end{corollary}

\begin{proof}  For any $m\ge 1$, by the union bound and the  Chebyshev inequality, the above probability term is bounded above by $$\sum_{k=m}^{\lceil cm \rceil} \frac{k \mbox{Var}(\Theta_1)}{\delta^2 k^2 (\e(\Theta_1)^2)} \le \frac{c}{\delta^2} \frac{\mbox{Var}(\Theta_1)}{(\e(\Theta_1)^2)}  $$ which tends to $0$ as $N\to\infty$, by \eqref{NBrw-exp} and \eqref{NBrw-var}.  \end{proof}

\end{document}